\documentclass[11pt]{article}
\usepackage{bbm}
\usepackage{amsmath}
\usepackage{latexsym,amsmath,amsthm,amssymb,amsxtra,mathrsfs,times}

\usepackage{graphicx}
\input xy
\xyoption{all}

\theoremstyle{plain}

  \newtheorem{thm}{\textbf{Theorem~}}[section]
  \newtheorem{lem}[thm]{\textbf{Lemma~}}

\theoremstyle{definition}
  \newtheorem{exmp}{\textsl{Example~}}[section]
  \newtheorem{rem}{\textbf{Remark~}}

\makeatletter      % '@' is now a normail "letter" for TeX
\@addtoreset{equation}{section}
\renewcommand\theequation
  {\ifnum \c@section>\z@ \thesection.\fi \@arabic\c@equation}

\newcommand{\lf}[2]{\left(\frac{#1}{#2}\right)}

\title{{
  \bf
{Complete Solving of Explicit Evaluation of Gauss Sums in the Index 2 Case
 \footnote{This research is partly supported by the National Natural Science
Foundation of China (Grant No. 10990011) and the Ph.D. Programs
Foundation of Ministry of Education of China (Grant No.
20090002120013). }}}
 \\
 \large Dedicated to Professor Yuan Wang on the occasion of his 80th Birthday}
\author{Jing Yang
 \\
  \small Department of Mathematical Sciences, Tsinghua University, Beijing, 100084, China;\\
  \small \& Division of Mathematical Sciences, School of Physical and Mathematical Sciences, \\
  \small Nanyang technological University, 637371, Singapore\\
  \small Email: jingyang@math.tsinghua.edu.cn\\[3mm]
 Lingli Xia
 \footnote{Corresponding author.}
 \\
  \small Basic Courses Department of Beijing Union University, Beijing, 100101, China\\
  \small Email: lingli@buu.edu.cn }
\date{}

\begin{document}

 \maketitle

%\def\hang{\hangindent\parindent}
%\def\textindent#1{\indent\llap{[#1]\enspace}\ignorespaces}
%\def\re{\hang\textindent}
%\def\re{\par\hang\textindent}

%\vspace*{10mm} \bc{\heiti\zihao{3}前\,\,言}\ec \thispagestyle{empty}
%\markboth{前\,\,言}{前\,\,言}
\abstract{ Let $p$ be a prime number, $N$ be a positive integer such
that $\gcd(N,p)=1$, $q=p^f$ where $f$ is the multiplicative order of
$p$ modulo $N$.  And let $\chi$ be a primitive multiplicative
character of order $N$ over finite field $\mathbb{F}_q$. This paper
studies the problem of explicit evaluation of Gauss sums $G(\chi)$
in ``\textsl{index 2 case}" (i.e.
$[(\mathbb{Z}/N\mathbb{Z})^*:<{p}>]=2$). Firstly, the classification
of the Gauss sums in index 2 case is presented. Then, the explicit
evaluation of Gauss sums $G(\chi^\lambda)\ \
(1\leqslant{}\lambda\leqslant{}N-1)$ in index 2 case with order $N$
being general even integer (i.e. $N=2^{r}\cdot N_0$ where $r,N_0$
are positive integers and $N_0\geqslant{}3$ is odd.) is obtained.
Thus, combining with the researches before, the problem of explicit
evaluation of Gauss sums in index 2 case is completely solved.

\vskip3mm {\bf keywords:\ }{Gauss sum, Stickelberger's Theorem,
Stickelberger congruence, Davenport-Hasse lifting formula,
Davenport-Hasse product formula}}

%MSC(2000) 11L05； 11L20

\section{Introduction}
\par\hspace{1.5em}
Gauss sum is one of the most important and fundamental objects and
tools in number theory and arithmetical geometry.  The explicit
evaluation of Gauss sums is an important but difficult problem,
which has not only theoretical value in number theory and
arithmetical geometry, but also important practical applications in
computer science, information theory, combinatorics and experimental
designs.

In 1801, C. F. Gauss \cite{Gauss0} gave the first result of this
problem for quadratic Gauss sums over $\mathbb{F}_p$. More exactly,
he determined the sign of the quadratic Gauss sums. Let $N$ be the
order of Gauss sum (the definition is in section 2). For the
relatively small $N$, such as $N=3,4,5,6,8,12$, people also gave
researches by the properties of the cyclotomic fields with
relatively small degree. One can see the more details in
%\cite{B-E,B-E2}
the Chapter 4 of \cite{B-E-W} or \cite[\S9.12]{I-R}.

In another research direction, for the Gauss sums with relatively
large orders $N$, by the Galois Theory of cyclotomic fields, people
have evaluated Gauss sums in some cases, such as pure Gauss sum and
the ones in index 2 and 4 case. For pure Gauss sums, i.e. the case
that $-1\in<{p}>\subset(\mathbb{Z}/N\mathbb{Z})^*$, Stickelberger
\cite{Stickelbg} gave an evaluation of Gauss sums $G(\chi)$ in 1890.
(Also see \cite[\S11.6]{B-E-W}, \cite[Thm5.16]{L-N} and
Lemma\ref{pureG} of this paper.)

In 1970's--2000's, for the Gauss sums of index 2, i.e. the case that
$-1\not\in<{p}>\subset(\mathbb{Z}/N\mathbb{Z})^*$ and
$[(\mathbb{Z}/N\mathbb{Z})^*:<{p}>]=2$, a series explicit
evaluations of Gauss sums have been given. In this case, the order
$N$ of Gauss sum $G(\chi)$ has no more than 2 distinct odd primes
factors.  In 1970's, R. J. McEliece \cite{McE} gave the evaluation
of Gauss sums in index 2 case for $N=l$\ \ ($l$ is odd prime) and
applied this to determine the (Hamming) weight distribution of some
irreducible cyclic codes. In 1992, M. Van Der Vlugt \cite{Vlugt}
gave the evaluation of Gauss sums in index 2 case for $N=l_1l_2$\ \
($l_1$ and $l_2$ are distinct odd primes). Similarly, the result was
applied to calculate the Hamming weight distribution of some
irreducible cyclic codes. (For the details on the relationship
between Hamming weight distribution of irreducible cyclic codes and
Gauss sums, we refer to
% \cite{McE-R,B-McE,B-M,M-S,MacW-S,Lint,Rm,Moisio} or
  \cite{B-McE,B-M,MacW-S,Lint,Moisio} or
 \cite[\S11.7]{B-E-W}.)
In 1997, P.Langevin \cite{La}, as generalization of \cite{McE}, gave
the evaluation of Gauss sums in index 2 case for $N=l^r$\ \ ($l$ is
odd prime, $r\geqslant{}1$). One year later, O. D. Mbodj \cite{M},
as generalization of \cite{Vlugt}, gave the evaluation of Gauss sums
in index 2 case for $N=l_1^{r_1}l_2^{r_2}$\ \ ($l_1,l_2$ are
distinct odd primes, $r_1,r_2\geqslant{}1$). For $N$ being power of
2, i.e. $N=2^t,\ \ (t\geqslant{}3)$, P. Meijer and M. van der Vlugt
\cite{M-V}, in 2003, evaluated the Gauss sums in index 2 case and
applied the results of Gauss sums to solve the problem of
calculating the number of rational points for some algebraic curves.
Since 2005, K. Feng, S. Luo and J. Yang \cite{Y1,Y2,Y3} have given
explicit evaluation of Gauss sums in index 4 case for $N$ being odd
and power of 2.

Up till now, there is no any work to study the Gauss sums in index 2
or 4 case with $N$ being ``general even number", i.e. $N=2^{r}\cdot
N_0$ where $r,\;N_0$ are positive integers and $ N_0\geqslant{}3$ is
odd. In this paper, we list all the classifications of index 2 case,
and give explicit evaluation of Gauss sums with $N$ being general
even number. Thus, combining with the previous papers, the problem
of explicit evaluation of Gauss sums in index 2 case is completely
solved.

This paper is organized as follow. Firstly, in section 2, we
introduce the preliminaries we need including the definitions and
several famous formulas about Gauss sums. In section 3, we present
all the classifications of index 2 case. More exactly, we list six
subcase A, B, C, D, E and F according to the factorization of $N$.
Then, in section 4.1, we present and prove the explicit formulas of
Gauss sums $G(\chi)$ in the later three subcases (Case D, E and F).
Finally, in section 4.2, we give the evaluation of Gauss sums
$G(\chi^\lambda)\ \ (1\leqslant{}\lambda\leqslant{}N-1)$ in all the
six subcases.

\vskip0cm

\section{Preliminaries}

\hspace{1.5em}Let $p$ be a prime number, $N\geqslant{}2$ be an
integer such that $(N,p)=1$. Let $f$ be the multiplicative order of
$p$ modulo $N$, denote by $f={\rm ord}_N(p)$, i.e. $f$ is the
smallest positive integer such that $p^f{\equiv}1\pmod{N}$. Take
$q=p^f$ and $\chi$ be a primitive multiplicative character of order
$N$ over $\mathbb{F}_q$, $T$ be the trace map from $\mathbb{F}_q$ to
$\mathbb{F}_p$. Then, for
$1\leqslant{}\lambda\leqslant{}N-1,1\leqslant{}\mu\leqslant{}p-1$,
the Gauss sum over $\mathbb{F}_q$ is defined as
 \begin{equation}\label{def}G(\chi^\lambda,\mu):=\sum\limits_{x\in\mathbb{F}_q}\chi(x)\zeta_p^{T(\mu x)}, \end{equation}
where $\zeta_p=\exp(2\pi i/p)$ is complex primitive $p$-th root of
1. When $(\lambda,N)=1$, $G(\chi^\lambda,\mu)$ is called {Gauss sum}
of order $N$. $N$ is called the order of $G(\chi^\lambda,\mu)$.

Since
$G(\chi^\lambda,v)=\overline{\chi^\lambda(v)}G(\chi^\lambda,1)$, we
can just consider $G(\chi^r,1)$, which is denoted as
$G(\chi^\lambda)$ for simplicity. Generally, $G(\chi)^\lambda$ is
belong to the ring of integers of cyclotomic field
$\mathbb{Q}(\zeta_{Np})=\mathbb{Q}(\zeta_N,\zeta_p)$. As we known,
the Galois group ${\rm
Gal}\left(\mathbb{Q}(\zeta_N,\zeta_p)\big{/}\mathbb{Q}\right)$ is
isomorphic to group $\left(\mathbb{Z}\big{/}Np\mathbb{Z}\right)^*$\
\
$\cong\left(\mathbb{Z}\big{/}N\mathbb{Z}\right)^*\times\left(\mathbb{Z}\big{/}p\mathbb{Z}\right)^*$.
More exactly,
$${\rm Gal}\left(\mathbb{Q}(\zeta_N,\zeta_p)\big{/}\mathbb{Q}\right)=\{\sigma_l\tau_t|l\in\left(\mathbb{Z}\big{/}N\mathbb{Z}\right)^*,t\in\left(\mathbb{Z}\big{/}p\mathbb{Z}\right)^*\},$$
where
$$\begin{array}{ll}\sigma_l(\zeta_N)=\zeta_N^l&\sigma_l(\zeta_p)=\zeta_p,\\\tau_t(\zeta_N)=\zeta_N&\tau_t(\zeta_p)=\zeta_p^t.\end{array}$$

 The following Lemma shows several basic results for Gauss sums. For more details, we refer to
 {\cite[\S8.2]{I-R}, and \cite[\S5.2]{L-N}}.

 \begin{lem}\label{thm0-2}For $l\in\left(\mathbb{Z}\big{/}N\mathbb{Z}\right)^*,\ t\in\left(\mathbb{Z}\big{/}p\mathbb{Z}\right)^*$,
 Let $G(\chi)$ be Gauss sum of order $N$ over finite field $\mathbb{F}_q$. Then,
 \par (1) if $\chi=\varepsilon$ (trivial character), $G(\chi)=-1$; otherwise, i.e. $\chi\neq\varepsilon$,
 $\vert G(\chi)\vert=\sqrt{q}$.
 \par (2) $\overline{G(\chi)}=\chi(-1)G(\bar{\chi})$, where $\bar{\chi}$ denotes the complex conjugation of $\chi$.
 \par (3) $\sigma_l\tau_t(G(\chi))=\bar{\chi}(t)G(\chi^l)$, especially, $\sigma_p(G(\chi))=G(\chi^p)=G(\chi)$. So $G(\chi)\in O_K[\zeta_p]$, where $K$ is the decomposition field of
 $p$ in $\mathbb{Q}(\zeta_N)$, i.e.  $K$ is the fixed subfield of $\sigma_p$ in $\mathbb{Q}(\zeta_N)$.
 \par (4) $G(\chi)^N\in O_K$, and $G(\chi)^s\big{/}G(\chi^s)\in
 O_K$ for each positive integer $s$ where $O_K$ is the ring of integers in K. \end{lem}

For the Gauss sum $G(\chi)$ of order $N$ over $\mathbb{F}_q$,
$G(\chi)^N\in\mathbb{Z}[\zeta_N]=O_M\ (M=\mathbb{Q}(\zeta_N))$ by
Lemma\ref{thm0-2}(4).  The Galois group
\[G=Gal(M/\mathbb{Q})=\{\sigma_a|1\leqslant a\leqslant N-1,(a,N)=1\}\ \ \ \ (\sigma_a(\zeta_N)=\zeta_N^a)\]
is canonically isomorphic to group $(\mathbb{Z}/N\mathbb{Z})^*$. So,
$(\mathbb{Z}/N\mathbb{Z})^*$ is often identified with $G$. A
profound result of Gauss sums was given by S. Stickerlberger
\cite{Stickelbg} in 1890, so called the \textsl{Stickelberger's
Theorem} (see \cite[\S 11.2]{B-E-W} or \cite[\S 11.3]{I-R}). It
reveals the prime ideal decomposition of $(G^N(\chi))O_K$. We note
that S. Stickelberger, actually, gave another more exact result, so
called ``\textsl{Stickelberger Congruence}" (see \cite[\S
11.2]{B-E-W}). And in the following text, we need it to determine
the sign (or unit root) ambiguities of Guass sums in some cases.

The first explicit evaluation of Gauss sums, for quadratic character
$\chi(x)=(\frac{x}{p})$ (the Legendre symbol) of $\mathbb{F}_p$, was
given by Gauss \cite{Gauss0}:
 \begin{equation}\label{G2}G(\chi)=\left\{\begin{array}{lll}
\sqrt{p},\ \ \ \mbox{if}\ p\equiv1 \pmod 4;\\
i\sqrt{p},\ \ \ \mbox{if}\ p\equiv3 \pmod 4.\end{array}\right.
 \end{equation}
This result can be generalized to quadratic Gauss sums over
$\mathbb{F}_q$ for any prime-power $q$ ($2\nmid q$) by
\textsl{Davenport-Hasse (Lift) Theorem} (see \cite[\S11.5]{B-E-W} or
\cite[Thm5.14]{L-N}). More exactly, let $\chi'=\chi\circ\mathrm{N}$.
Then, the corresponding quadratic Gauss sums over
$\mathbb{F}_q=\mathbb{F}_{p^d}~$ are given by
 \begin{equation}\label{def1-3}G(\chi')=\left\{\begin{array}{cl}(-1)^{d-1}\sqrt{q}&\mbox{if~}p{\equiv}1\pmod{4};\\-(-i)^d\sqrt{q}&\mbox{if~}p{\equiv}3\pmod{4}.\end{array}\right.\end{equation}

After Gauss's result on $N=2$, using arithmetic properties on field
$\mathbb Q(\zeta_N)$, the value of Gauss sums $G(\chi)$ with
relatively small order $N$ have been determined explicitly. See
\cite[\S9.12]{I-R} for cubic and quartic Gauss sums ($N=3,4$) and
Chapter 5 of \cite{B-E-W} for more other cases.

\vskip2mm On the other hand, for the Gauss sums with relatively
large order $N$, by the Galois Theory of cyclotomic field, people
have evaluated Gauss sums in some cases. Such as, when $-1$ is a
power of $p$ in $(\mathbb{Z}/N\mathbb{Z})^*$, i.e.
$-1\in<{p}>\subset(\mathbb{Z}/N\mathbb{Z})^*$, Gauss sums $G(\chi)$
can be determined by the following result, which are called
``self-conjugate" or ``pure" Gauss sums.

%
%关于Gauss和的显式计算问题，目前学术界有两个研究方向：一是当Gauss和的次数较小时，利用低次数域相对简单的算术性质，决定Gauss和的明显表达式；
%另一个是通过Galois理论，分析分圆域及其子域的算术性质来分析计算Gauss和。下面我们简述该问题的研究历史。

 \begin{lem}[See {\cite[\S 11.6]{B-E-W}}~or~{\cite[Thm5.16]{L-N}}]\label{pureG}
 Suppose that $\chi$ is a multiplicative
character of order $N$ over $\mathbb{F}_q$, $q=p^f$.
%and $N\nmid (p-1)$.
Assume that there exists an integer $t\geq1$ such that $p^t\equiv-1
\pmod m$, with $t$ chosen minimal. Then $f=2ts$ for some positive
integer $s$, and
\[G(\chi)=\left\{\begin{array}{ll}
(-1)^{s-1}\sqrt{q},&\mbox{if}\,p=2,\\
(-1)^{s-1+(p^t+1)s/N}\sqrt{q},&\mbox{if}\,p\geqslant3.\end{array}\right.\]\hfill{$\square$}
 \end{lem}

From now on we assume that $-1\not\in\langle
p\rangle\subset(\mathbb{Z}/N\mathbb{Z})^*$ so that $K$ (defined in
Lemma\ref{thm0-2}(3)) is an imaginary abelian field of degree $r$,
where
 $$r=[(\mathbb{Z}/N\mathbb{Z})^*:<p>]=\varphi(N)/f,$$
$\varphi(\cdot)$ is Euler function. It is called the ``\textsl{index
$r$ caes}''. In 1970's--2000's, the Gauss sums in \textsl{index
$r=$2 case} have been studied and evaluated explicitly in a series
of papers \cite{McE,Vlugt,La,M,M-V}. And since 2005, the case of
\textsl{index $r=4$} has been studied in papers \cite{Y1,Y2,Y3}.

\begin{lem}\label{lem0-1} Suppose that $\chi$ is a multiplicative
character of order $N$ over $\mathbb{F}_q$, and let $T$ be the trace
map from $\mathbb{F}_q$ onto $\mathbb{F}_p$. And
$K\subset\mathbb{Q}(\zeta_N)$ is the invariant subfield for
$\sigma_p$, Then
\begin{equation}\label{def1-6}G(\chi)=\left(\sum\limits_{x\in\mathbb{F}_q\atop
T(x)=1}\chi(x)\right)\left(\sum\limits_{y\in\mathbb{F}_p}\chi(y)\zeta_p^{y}\right)\end{equation}
 and $\displaystyle\sum\limits_{x\in\mathbb{F}_q\atop T(x)=1}\chi(x)\in O_K$.
 \end{lem}

 \begin{proof}
 Since $$G(\chi)=\sum\limits_{x\in\mathbb{F}_q\atop
 T(x)=0}\chi(x)+\sum\limits_{a=1}^{p-1}\sum\limits_{x\in\mathbb{F}_q\atop
 T(x)=1}\chi(ax)\zeta_p^{a}$$
and $\chi$ is nontrivial on $\mathbb{F}_q$, we know that the first
summation of the formula above is equal to zero, while the second
summation is equal to the right side of
 (\ref{def1-6}).
 Finally, since $f(x)=\sum\limits_{x\in\mathbb{F}_q\atop T(x)=1}\chi(x)\in O_M$,\ $M=\mathbb{Q}(\zeta_N)$, and
 $$\sigma_p(f(x))=\sum\limits_{x\in\mathbb{F}_q\atop T(x)=1}\chi^p(x)=\sum\limits_{x\in\mathbb{F}_q\atop T(x)=1}\chi(x^p)=\sum\limits_{y\in\mathbb{F}_q\atop T(y)=1}\chi(y)=f(x),$$
we have $f(x)\in O_K$.
 \end{proof}

Let $\chi|_{\mathbb{F}_p}$ denote the restriction of $\chi$ onto
$\mathbb{F}_p$. From Lemma \ref{lem0-1} we know that
 \begin{equation}\label{def1-5}G(\chi)=\left(\sum\limits_{x\in\mathbb{F}_q\atop T(x)=1}\chi(x)\right)G_p(\chi).\end{equation}
where $G_p(\chi)=G(\chi|_{\mathbb{F}_p})$. When the order of
$\chi|_{\mathbb{F}_p}$ is relatively small, we can calculate
$G_p(\chi)$ by the results of the Gauss sum with a relatively small
order over $\mathbb{F}_p$. For the order of $\chi|_{\mathbb{F}_p}$,
there exists the following result.

 \begin{lem}[see {\cite[Prop 11.4.1]{B-E-W}}]\label{lem0-2}
 Assume that $\chi$ is a multiplicative character of order $N$ over $\mathbb{F}_q=\mathbb{F}_{p^f}$. Then the order of the restriction of $\chi$ onto $\mathbb{F}_p$ is：
 $$\displaystyle\frac{N}{~~(N,\frac{p^f-1}{p-1})~~}$$
 \end{lem}

Particularly, when $\chi|_{\mathbb{F}_p}$ is a quadratic character,
the evaluation of the Gauss sums $G(\chi)$ is directly reduced to
the evaluation of the summation
$\displaystyle\sum\limits_{x\in\mathbb{F}_q\atop T(x)=1}\chi(x)\in
 O_K$ by formula (\ref{G2}).

%一般情况下，$G(\chi)\in\mathbb{Q}(\zeta_N,\zeta_p)$。由引理\ref{thm0-2}(3)以及下一个引理，我们知道Gauss和$G(\chi)$实际上是属于域$K(\zeta_p)$的整元素。
%$$\xymatrix{
%   & \ar@{-}[dl]\mathbb{Q}(\zeta_n,\zeta_p) \ar@{-}[dr] &\\
% \mathbb{Q}(\zeta_N)\ar@{-}[d]& &K(\zeta_p)\ar@{-}[d]\ar@{-}[dll] \\
% K\ar@{-}[dr]& & \mathbb{Q}(\zeta_p)\ar@{-}[dl]\\
% & \mathbb{Q} & }$$

\section{Classification of order $N$ in index 2 case}

\hspace{1.5em}We always keep these assumptions in following text.
Assume that
\par
(I).\ $p$ is a prime number, $N\geqslant2,\ (p,N)=1$, the order of
$p$ modulo $N$ is $f=\frac{\varphi(N)}{2}$, so that
$[(\mathbb{Z}/N\mathbb{Z})^*:<p>]=2$ and the decomposition field $K$
of $p$ in $\mathbb Q(\zeta_N)$ is a quadratic abelian field.
\par
(II).\ $q=p^f$, $\chi$ is a multiplicative character of $\mathbb
F_q$ with order $N$, $G(\chi)$ is the Gauss sum of order $N$ over
$\mathbb F_q$ defined by (\ref{def}).
\par
(III).\ $-1\not\in<{p}>\subset(\mathbb{Z}/N\mathbb{Z})^*$, so
 $K$ is an imaginary field.
\par
In this section, we will determine all the possibilities of $N$
satisfying assumptions (I),(II) and (III), and also determine the
type of the corresponding imaginary quadratic subfield $K$ of
$\mathbb{Q}(\zeta_N,\zeta_p)$.

\vskip2mm Suppose that $N$ has the prime factorization
$$N=2^{r_0}l_1^{r_1}\cdots l_s^{r_s},$$
where $s\geqslant{}0$, $l_i$ are distinct odd primes, $r_i$ are
non-negative integers $(0\leqslant{}i\leqslant{}s)$. And suppose
that any two of $\varphi(l_i^{r_i})\ \ (1\leqslant{}i\leqslant{}s)$
have no odd common prime factors. By the Chinese Remainder Theorem,
we have
 \begin{equation}\label{CRT}(\mathbb{Z}/N\mathbb{Z})^*\cong (\mathbb{Z}/2^{r_0}\mathbb{Z})^*\times(\mathbb{Z}/l_1^{r_1}\mathbb{Z})^*\times\cdots\times(\mathbb{Z}/l_s^{r_s}\mathbb{Z})^*.\end{equation}
Let
$\frac{\varphi(2^{r_0})}{a_0},\frac{\varphi(l_1^{r_1})}{a_1},\cdots,\frac{\varphi(l_s^{l_s})}{a_s}$
be the order of $p$, respectively, in each subgroup in right-side of
(\ref{CRT}), where $a_0,a_1,\cdots,$~$a_s\in\mathbb{N},
a_0\mid\varphi(2^{r_0}), a_i\mid\varphi(l_i^{r_i})$ for
$1\leqslant{}i\leqslant{}s$. Then
 \begin{equation}\label{(1)}f=\frac{\varphi(N)}{2}=\frac{1}{2}\varphi(2^{r_0})\varphi(l_1^{r_1})\cdots\varphi(l_s^{r_s})=\left[\frac{\varphi(2^{r_0})}{a_0},\frac{\varphi(l_1^{r_1})}{a_1},\cdots,\frac{\varphi(l_s^{r_s})}{a_s}\right].\end{equation}
By the Chinese Remainder Theorem again, there are the primitive
roots $g_j$ modulo $l_{j}^{r_j}\ \ (1\leqslant j\leqslant s)$ and
primitive root $g_0$ modulo $2^{r_0}$ such that
 \begin{equation}\label{(2)}\left\{\begin{array}{l}
 g_0{\equiv}1\pmod{l_\lambda^{r_\lambda}},\quad \mbox{for}\ ~1\leqslant{}\lambda\leqslant{}s;\\
 g_j\equiv1 \pmod{2^{r_0}},\ g_j\equiv1 \pmod{l_\lambda^{r_\lambda}},\quad\mbox{for}\ ~1\leqslant\lambda\leqslant s,\ \lambda\not=j;\\
 p\equiv g_0^{a_0}g_1^{a_1}\cdots g_s^{a_s} \pmod N.
\end{array}\right.
\end{equation}

When $r_0=0$, $N$ is odd. So, by (\ref{(1)}), $N$ has no more than 2
odd prime factors, i.e. $s\leqslant{}2$. Then we have two subcases
according to $N$ having one odd prime factor or two odd prime
factors, where $g_1,\; g_2$ can be odd or even.

 \vskip2mm
\underline{\textbf{Case A}}.\quad $N=l_1^{r_1},\ l_1$ be odd ,
$r_1\geqslant{}1,\ p{\equiv}g_1^2\pmod{l_1^{r_1}}$. The assumption
(III) $-1\not\in <{p}>~\Leftrightarrow~-1$ is quadratic non-residue
in $(\mathbb{Z}/N\mathbb{Z})^*$\ $\Rightarrow$~$l_1{\equiv}3\pmod4,\
K=\mathbb{Q}(\sqrt{-l_1})$.

 \vskip2mm
\underline{\textbf{Case B}}.\quad $N=l_1^{r_1}l_2^{r_2},\
\frac{1}{2}\varphi(l_1^{r_1})\varphi(l_2^{r_2})=\frac{\varphi(l_1^{r_1})\varphi(l_2^{r_2})}{a_1a_2}\Big/(\frac{\varphi(l_1^{r_1})}{a_1},\frac{\varphi(l_2^{r_2})}{a_2})$.
By (\ref{(1)}), we know $a_1a_2\leqslant{}2$. Then we have two
subcases of Case B:

\quad\underline{\textbf{Case B1}}.\quad $a_1a_2=1,\
p{\equiv}g_1g_2\pmod{N}$，$-1\not\in<{p}>\Leftrightarrow
\{l_1,l_2\}{\equiv}\{1,3\}\pmod4$. Without loss of generality, we
assume that $(l_1,l_2){\equiv}(3,1)\pmod4$, then
$<{p}>=<g_1^2,g_2^2>,\ K=\mathbb{Q}(\sqrt{-l_1l_2})$;

\quad\underline{\textbf{Case B2}}.\quad $a_1a_2=2$, and let
$p{\equiv}g_1^2g_2\pmod{N}$, then
$(\frac{l_1-1}{2},l_2-1)=1~\Rightarrow~l_1{\equiv}3\pmod4$,
$l_2{\equiv}1\pmod2$. $<{p}>=<g_1^2,g_2>,\
K=\mathbb{Q}(\sqrt{-l_1})$.

For simplicity of the following evaluation, we assume $l_1\neq3$ in
Case A and Case B2.

\vskip2mm When $s=0$, i.e. $N$ just has prime factor 2. We have the
following subcase for $r_0\geqslant{}3$. (Since when $r_0=2$ we have
$p{\equiv}3\pmod{4}$, and it's self-conjugate (pure) Gauss sum,
which can be determined by Theorem\ref{pureG}.)

\underline{\textbf{Case C}}.\quad $N=2^{r_0},\ r_0\geqslant{}3,\
p{\equiv}3~\mbox{or}~5\pmod8,\ f=2^{r_0-2},\
K=\mathbb{Q}(\sqrt{-2})$ or $\mathbb{Q}(\sqrt{-1})$.

 \vskip2mm
When $r_0=1$, let $N=2N_0$\ \ ($N_0$ is odd). Similarly, by
(\ref{(1)}), $s\leqslant{}2$. Since
$(\mathbb{Z}/N\mathbb{Z})^*\cong(\mathbb{Z}/N_0\mathbb{Z})^*$, we
have two subcases, Case D and E, respectively corresponding to Case
A and B, where $g_1\;g_2$ must be odd.

\underline{\textbf{Case D}}.\quad $N=2l_1^{r_1},\ N_0=l_1^{r_1},\
3\neq l_1{\equiv}3\pmod4,\ p{\equiv}g_1^2\pmod{N_0},\
K=\mathbb{Q}(\sqrt{-l_1})$.

 \vskip2mm
\underline{\textbf{Case E}}.\quad $N=2l_1^{r_1}l_2^{r_2}$. Similar
as Case B, we have two subcases of Case E according to $a_1a_2$
equals 1 or 2:

\quad\underline{\textbf{Case E1}}.\quad $a_1a_2=1,\
p{\equiv}g_1g_2\pmod{N_0}$，$-1\not\in<{p}>\Leftrightarrow
(l_1,l_2){\equiv}(3,1)\pmod4$，$<{p}>=<g_1^2,g_2^2>,\
K=\mathbb{Q}(\sqrt{-l_1l_2})$;

\quad\underline{\textbf{Case E2}}.\quad $a_1a_2=2$, and let
$p{\equiv}g_1^2g_2\pmod{N}$, $3\neq l_1{\equiv}3\pmod4$,
$l_2{\equiv}1\pmod2$, $<{p}>=<g_1^2,g_2>,\
K=\mathbb{Q}(\sqrt{-l_1})$.

 \vskip2mm
When $s\geqslant{}1,r_0\geqslant{}2$, we have $s=1,r_0=2$ by
(\ref{(1)}).

\underline{\textbf{Case F}}.\quad $N=4l_1^{r_1}$, $N_0=l_1^{r_1}$.
Then $a_0a_1\leqslant{}2$, and we have three subcase of Case F
according to the values of $a_0$ and $a_1$:

\quad\underline{\textbf{Case F1}}.\quad $a_0=a_1=1$,
$p{\equiv}g_0g_1\pmod{N}$, i.e. $p{\equiv}3\pmod4$ and
$p{\equiv}g_1\pmod{N_0}$,
$-1\not\in<{p}>~\Leftrightarrow~l_1{\equiv}1\pmod4$,
$<{p}>=<g_0g_1>$, $K=\mathbb{Q}(\sqrt{-l_1})$.

\quad\underline{\textbf{Case F2}}.\quad $a_0=2,\ a_1=1$,
$p{\equiv}g_0^2g_1$, i.e. $p{\equiv}1\pmod4,\ p{\equiv}g_1\pmod{N}$,
$<{p}>=<g^2_0,g_1>,\ K=\mathbb{Q}(\sqrt{-1})$.

\quad\underline{\textbf{Case F3}}.\quad $a_0=1,a_1=2,\
p{\equiv}3\pmod4,\ l_1{\equiv}3\pmod4,\ p{\equiv}g_0g_1^2\pmod{N}$,
$<{p}>=<g^2_1>$, $K=\mathbb{Q}(\sqrt{-l_1})$.

\section{Explicit evaluation of Gauss sums in index 2 case}

\subsection{Explicit evaluation of $G(\chi)$}

\hspace{1.5em} In this section, we give explicit evaluation of Gauss
sum $G(\chi)$ in each subcases (i.e. Case A, B, C, D, E, F), where
the results of Case A, B and C has been shown in previous papers.

\vskip2mm{\large \maltese \quad\underline{\textbf{Case
A}}.}\vskip2mm

$N=l_1^{r_1}$\ \ ($r_1\geqslant{}1$), the result was given by
P.Langevin \cite{La} in 1997.
 \begin{thm}\label{r=2-3}
Let $N=l_1^{r_1},~l_1{\equiv}3\pmod4,~l_1>3$,
$[(\mathbb{Z}/N\mathbb{Z})^*:<{p}>]=2$, i.e.
$f=\frac{\varphi(N)}{2}$,\ $q=p^f$ and let $\chi$ be a primitive
multiplicative character of order $N$ over $\mathbb{F}_q$. Then
\[G(\chi)=\frac{1}{2}p^{\frac{1}{2}(f-h_1)} (a+b\sqrt{-l_1}),\]
where $h_1=h(\mathbb{Q}(\sqrt{-l_1}))$ is the ideal class number of
field $\mathbb{Q}(\sqrt{-l_1})$ and $a,b\in\mathbb{Z}$ are
determined by
\begin{equation}\label{(a-1)}\left\{\begin{array}{l}4p^{h_1}=a^2+l_1b^2,\\
a\equiv-2p^{\frac{1}{2}(f+h_1)} \pmod
{l_1}.\end{array}\right.\end{equation}
\end{thm}
\begin{rem}\label{remark1}
 Since
$$p^{\frac{f+h_1}{2}}\Big/p^{\frac{\varphi(l_1)/2+h_1}{2}}=p^{\frac{\varphi(l_1^{r_1})-\varphi(l_1)}{4}}=\left(p^{\frac{l_1-1}{2}}\right)^{\frac{l_1^{r_1-1}-1}{2}}
 {\equiv}\left(g_1^{{l_1-1}}\right)^{\frac{l_1^{r_1-1}-1}{2}}{\equiv}1\pmod{l_1},$$
the second equation of (\ref{(a-1)}) implies
$a{\equiv}-2p^{\frac{l_1-1+2h_1}{4}}\pmod{l_1}$. Therefore,
equations (\ref{(a-1)}) are equivalent to
\begin{equation}\label{(aa)}\left\{\begin{array}{l}4p^{h_1}=a^2+l_1b^2,\\a{\equiv}2p^{\frac{l_1-1+2h_1}{4}}\pmod{l_1}.\end{array}\right.\end{equation}
From equations (\ref{(aa)}), one can find that the sign of $a$ is
just relational with $l_1,\ p$, however, is not relational with
$r_1$. So, we always take the principal ideal
$\wp^{h_1}=(\frac{a+b\sqrt{-l_1}}{2})O_{\mathbb{Q}(\sqrt{-l_1})}$ in
the following text, where $l_1{\equiv}3\pmod4,\ \wp$ is a prime
ideal factor of $p$ in the integral ring of
$\mathbb{Q}(\sqrt{-l_1})$ and integers $a,b$ are determined by
equations (\ref{(aa)}). \end{rem}

\vskip2mm{\large \maltese \quad\underline{\textbf{Case
B}}.}\vskip2mm

$N=l_1^{r_1}l_2^{r_2}$\ \ ($l_1,l_2$are different odd prime numbers,
$r_1,r_2\geqslant{}1$). In 1998, O. D. Mbodj \cite{M} gave
evaluation of the Gauss sums $G(\chi)$ in Case B.

 \begin{thm}\label{r=2-4}
Let $N=l_1^{r_1}l_2^{r_2}$, $[(\mathbb{Z}/N\mathbb{Z})^*:<{p}>]=2$,
i.e. $f=\frac{\varphi(N)}{2}$, $q=p^f$, and let $\chi$ be a
primitive multiplicative character of order $N$ over $\mathbb{F}_q$.
Assume that the orders of $p$ in groups
$(\mathbb{Z}/l_1^{r_1}\mathbb{Z})^*$ and
$(\mathbb{Z}/l_2^{r_2}\mathbb{Z})^*$ are respectively
$\varphi(l_1^{r_1})/a_0$ and $\varphi(l_2^{r_2})/a_1$. Then
 \par
(i).~for Case B1, ($a_1=a_2=1,\ (l_1,l_2){\equiv}(1,3)\pmod4$)
\[G(\chi)=\frac{1}{2}p^{\frac{1}{2}(f-h_{12})}(a'+b'\sqrt{-l_1l_2}),\]
where $h_{12}=h(\mathbb{Q}(\sqrt{-l_1l_2}))$ and integers $a',b'$
are determined by equations
 \begin{equation}\label{(b1)}\left\{\begin{array}{l}4p^{h_{12}}=(a')^2+l_1l_2(b')^2,\\ a'\equiv2p^{\frac{1}{2}h_{12}} \pmod {l_1};\end{array}\right.\end{equation}
\par
(ii).~For Case B2, ($a_1=2,\ a_2=1,\ 3\neq l_1{\equiv}3\pmod4$)
\[G(\chi)=\left\{\begin{array}{lll}
p^{\frac{f}{2}},&\mbox{if~}\,\left(\frac{l_2}{l_1}\right)=1,\\
p^{\frac{f}{2}-h_1}(\frac{a+b\sqrt{-l_1}}{2})^2,&\mbox{if~}\,\left(\frac{l_2}{l_1}\right)=-1.\end{array}\right.\]
where $h_1=h(\mathbb{Q}(\sqrt{-l_1}))$, integers $a,b$ are
determined by equations (\ref{(aa)}) and $\lf{\cdot}{l_1}$ is
Legendre Symbol modulo $l_1$. \end{thm}

\vskip2mm{\large \maltese \quad\underline{\textbf{Case
C}}.}\vskip2mm

For $N=2^{r_0}\ (r_0\geqslant{}3)$, P. Meijer and M. van der Vlugt
\cite{M-V} gave evaluation of the Gauss sums in index 2 case, in
2003. When $N=2^{r_0}\ (r_0\geqslant{}3)$, it is known from
elementary number theory that the primes $p$ such that
$p{\equiv}3,5\pmod{8}$ are exactly the primes which generate
subgroup of index 2 in $(\mathbb{Z}/N\mathbb{Z})^*$. So,

 \begin{thm}\label{r=2-5}
Let $N=2^{r_0}\ (r_0\geqslant{}3)$, $\chi$ be a primitive
multiplicative character of order $N$ over $\mathbb{F}_q$ and $P_1$
be a prime ideal factor of $p$ in the integral ring of
 $\mathbb{Q}(\zeta_N)$. Then
 \par
(i).~When $p{\equiv}3\pmod8$,
$$G(\chi)=G(\chi_{_{P_1}})=\varepsilon_1i\sqrt{p}p^{2^{r_0-3}-1}(a+ib\sqrt2),$$
where $a,b\in\mathbb{Z}$ are determined by
$(a+ib\sqrt2)=P_1\cap\mathcal{O}(i\sqrt2)$ and $\varepsilon_1=\pm1$
can be solved by Stickelberger congruence.
 \par
(ii).~ When $p{\equiv}5\pmod8$,
$$G(\chi)=G(\chi_{_{P_1}})=\varepsilon_2p^{2^{r_0-3}}\sqrt{a+ib}\big{/}\sqrt[4]{p},$$
where $a,b\in\mathbb{Z}$ are determined by
$(a+ib)=P_1\cap\mathbb{Z}[i]$, $\sqrt{a+ib}$ has a positive real
part, and $\varepsilon_2\in\{\pm1,\pm i\}$ can be solved by
Stickelberger congruence.\end{thm}

\vskip2mm{\large \maltese \quad\underline{\textbf{Case
D}}.}\vskip2mm

When $N=2l_1^{r_1}=2N_0$, since $(N,p)=1$, $p$ must be odd prime. $\
3\neq l_1{\equiv}3\pmod4,\
f=\frac{\varphi(N)}{2}=\frac{(l_1-1)}{2}l_1^{r_1-1}{\equiv}1\pmod2,\
q=p^f=2N_0n+1\ (\exists n\in\mathbb{Z})$. $g_1$ is defined by
(\ref{(2)}) in Section 1, i.e. $g_1$ is odd primitive root modulo
$N_0$. Then $g_1$ is also the primitive root modulo $N$, and we can
take $p{\equiv}g_1^2\pmod{N}$.

In Case D, $\chi$ is a primitive multiplicative character of order
$N=2l_1^{r_1}$ over $\mathbb{F}_q$, which means that $\chi^2$ is the
character of order $N_0=l_1^{r_1}$, and since $f={\rm
ord}_{N}(p)={\rm ord}_{N_0}(p)$, $\chi^2$ is primitive. By the
result of Case A,
$$G(\chi^2)=p^{\frac{f-h_1}{2}}\left(\frac{a+b\sqrt{-l_1}}{2}\right),$$
where integers $a,b$ are determined by equations (\ref{(aa)}). By
Darvenport-Hasse product formula (\cite[\S 11.3]{B-E-W}), we have
that
$$G(\chi^2)=\chi^2(2)\dfrac{G(\chi) G(\chi^{l_1^{r_1}+1})}{G(\chi^{l_1^{r_1}})}=\dfrac{G(\chi) G(\chi^{2\frac{1+l_1^{r_1}}{2}})}{(\sqrt{p^*})^f}
 =\left\{\begin{array}{ll}\frac{G(\chi)}{(\sqrt{p^*})^f}G(\chi^2)&\mbox{if $\frac{1+l_1^{r_1}}{2}\in R_2$}\\\frac{G(\chi)}{(\sqrt{p^*})^f}\overline{G(\chi^2)}&\mbox{if $\frac{1+l_1^{r_1}}{2}\in \overline{R}_2$},\end{array}\right.$$
where $R_2,\ \overline{R}_2$ denote respectively as the sets of
quadratic remainder and quadratic non-remainder modulo $l_1$. Then
 \begin{equation}\label{(a-3)}G(\chi)=\left\{\begin{array}{ll}(\sqrt{p*})^f=(-1)^{\frac{f-1}{2}\frac{p-1}{2}}\sqrt{p^*}p^{\frac{f-1}{2}}&\mbox{if $\lf{(l_1^{r_1}+1)/2}{l_1}=1$}\\
 \frac{(G(\chi^2))^2(\sqrt{p*})^f}{p^f}=(-1)^{\frac{f-1}{2}\frac{p-1}{2}}\sqrt{p^*}p^{\frac{f-1}{2}-h_1}(a+b\frac{-1+\sqrt{-l_1}}{2})^2&\mbox{if
 $\lf{(l_1^{r_1}+1)/2}{l_1}=-1$.}\end{array}\right.\end{equation}
Since
$$\lf{(l_1^{r_1}+1)/2}{l_1}=\frac{\lf{l_1^{r_1}+1}{l_1}}{\lf{2}{l_1}}=\lf{1}{l_1}\lf{2}{l_1}=\left\{\begin{array}{cl}1&\mbox{if $l_1{\equiv}7\pmod8$};\\-1&\mbox{if $l_1{\equiv}3\pmod8$},\end{array}\right.$$
we have got the formula of Gauss sums $G(\chi)$ in Case D as follow:
 \begin{thm}[Case D]\label{thm-d}
Let $N=2l_1^{r_1},~l_1{\equiv}3\pmod4,~l_1>3$,
$[(\mathbb{Z}/N\mathbb{Z})^*:<{p}>]=2$, i.e.
$f=\frac{\varphi(N)}{2}$, $q=p^f$, and let $\chi$ be a primitive
multiplicative character of order $N$ over $\mathbb{F}_q$. Then
$$G(\chi)=\left\{\begin{array}{ll}(-1)^{r_1\cdot\frac{p-1}{2}}\sqrt{p^*}p^{\frac{f-1}{2}},&\mbox{if $~l_1{\equiv}7\pmod8$;}\\
    (-1)^{(r_1+1)\cdot\frac{p-1}{2}}\sqrt{p^*}p^{\frac{f-1}{2}-h_1}(\frac{a+b\sqrt{-l_1}}{2})^2,&\mbox{if $~l_1{\equiv}3\pmod8$,}\end{array}\right.$$
where $h_1$ is the ideal class number of $\mathbb{Q}(\sqrt{-l_1})$
and integers $a,b$ are determined by equations (\ref{(aa)}).
\end{thm}

\vskip2mm \begin{exmp}(For \underline{Case D})
\par (1).\ Let $l_1=7,\ N=14,\
f=\frac{\varphi(N)}{2}=3$. The primitive roots modulo 7 are $2,3$,
then $p{\equiv}g_1^2{\equiv}2,4\pmod7$. Take $p=11$, then
$<{p}>=\{1,11,9\},\ g_1<{p}>=\{3,5,13\}$. By Theorem\ref{thm-d}, we
know that the Gauss sum $G(\chi)$ of order 14 over
$\mathbb{F}_{11^3}$ is
$$G(\chi)=-11\sqrt{-11}.$$
\par (2).\ Let $l_1=11,\ N=22,\
f=\frac{\varphi(22)}{2}=5$. The primitive roots modulo 11 are
$2,6,7,8$, then $p{\equiv}g_1^2{\equiv}4,3,5,9\pmod{11}$. Take
$p=3$, so the equations $\left\{\begin{array}{l}a^2+11\cdot
b^2=3\cdot4\\a{\equiv}-2\cdot3^{\frac{5+1}{2}}\pmod{11}\end{array}\right.$
have solutions
$\left\{\begin{array}{l}a=1\\b=\pm1\end{array}\right.$. By
Theorem\ref{thm-d}, we have the Gauss sum $G(\chi)$ of order 22 over
$\mathbb{F}_{3^5}$ and its conjugation $\overline{G(\chi)}$ are
$$\{G(\chi),\overline{G(\chi)}\}=\{2\sqrt{-3}\left(\frac{-5\pm\sqrt{-11}}{2}\right)\}.$$
\end{exmp}

\vskip2mm{\large \maltese \quad\underline{\textbf{Case
E}}.}\vskip2mm

%我们在情形E中用相对简单的方法二，即用Davenport-Hasse乘积公式来计算Gauss和。而用方法一，也可得到一致的结果。由于篇幅原因，就不再累述。

Let $\chi$ be a primitive multiplicative character of order
$N=2N_0=2l_1^{r_1}l_2^{r_2}$ over $\mathbb{F}_q$, $f={\rm
ord}_N(p)=\varphi(N)/2$\ (must be even). By
$(\mathbb{Z}/N\mathbb{Z})^*\cong(\mathbb{Z}/N_0\mathbb{Z})^*$, we
know that $\chi^2$ is the primitive character of order $N_0$ over
$\mathbb{F}_q$. Then $G(\chi^2)$ can be evaluate by Theorem
\ref{r=2-4} in Case B. By Davenport-Hasse product formula, we have
that
$$G(\chi^2)=\chi^2(2)\dfrac{G(\chi) G(\chi^{N_0+1})}{G(\chi^{N_0})},$$
where $\chi^{N_0}$ is quadratic character over $\mathbb{F}_q$. By
formula (\ref{def1-3}),
$G(\chi^{N_0})=-(\sqrt{p^*})^f=(-1)^{\frac{f}{2}\frac{p-1}{2}+1}p^{\frac{f}{2}}$.
And since $(N_0+1,2N_0)=2$,
$$G(\chi^{N_0+1})=G(\chi^{2\frac{N_0+1}{2}})=\left\{\begin{array}{ll}G(\chi^2)&\mbox{if\ $\frac{N_0+1}{2}\in<{p}>$}\\\overline{G(\chi^2)}&\mbox{if $\frac{N_0+1}{2}\in-<{p}>$}.\end{array}\right.$$
Then
$$G(\chi)=\left\{\begin{array}{ll}(-1)^{\frac{p-1}{2}\frac{f}{2}+1}\chi^2(2)p^{\frac{f}{2}}&\mbox{if $\frac{N_0+1}{2}\in<{p}>$;}\\
      (-1)^{\frac{p-1}{2}\frac{f}{2}+1}\chi^2(2)p^{-\frac{f}{2}}(G(\chi^2))^2&\mbox{if $\frac{N_0+1}{2}\in-<{p}>$.}  \end{array}\right.$$

\vskip2mm For \underline{\textbf{Case E1}}.
$p{\equiv}g_1g_2\pmod{N_0},\ (l_1,l_2){\equiv}(3,1)\pmod4,\ 4\mid
f=\varphi(N_0)/2$, $K=\mathbb{Q}(\sqrt{-l_1l_2})$. Since
$a\in<{p}>~\Leftrightarrow~\lf{a}{l_1l_2}=1$,
$\frac{N_0}{2}\in<{p}>~\Leftrightarrow~\lf{\,\frac{N_0+1}{2}\,}{l_1l_2}=1~\Leftrightarrow~\lf{2}{l_1l_2}=1~\Leftrightarrow~l_1l_2{\equiv}7\pmod8$;
and $\frac{N_0}{2}\in-<{p}>~\Leftrightarrow~l_1l_2{\equiv}3\pmod8$.
Similar as Lemma\ref{lemd1}, we have that $(p-1,N_0)=1$ and the
order of character $\chi\big|_{\mathbb{F}_p}$ is
$$\frac{N}{(N,\frac{q-1}{p-1})}=\frac{2}{(2,\frac{2n}{p-1})}=1.$$
Therefore, by Theorem\ref{r=2-4}(i),
$$G(\chi)=\left\{\begin{array}{ll}-p^{\frac{f}{2}}&\mbox{if $l_1l_2{\equiv}7\pmod8$;}\\-p^{\frac{f}{2}-h_{12}}(a'+b'\frac{-1+\sqrt{-l_1l_2}}{2})^2&\mbox{if $l_1l_2{\equiv}3\pmod8$,}\end{array}\right.$$
where $h_{12}=h(\mathbb{Q}(\sqrt{-l_1l_2}))$ and integers $a',b'$
are determined by equations (\ref{(b1)}) in Theorem\ref{r=2-4}.

\vskip2mm For \underline{\textbf{Case E2}}.
$p{\equiv}g_1^2g_2\pmod{N_0},\
l_1{\equiv}3\pmod4,l_2{\equiv}1\pmod2,\
K=\mathbb{Q}(\sqrt{-l_1l_2})$, and $f=\varphi(N)/2{\equiv}$
$\left\{\begin{array}{ll}0\pmod4&\mbox{if
$l_2{\equiv}1\pmod4$};\\2\pmod4&\mbox{if
$l_2{\equiv}3\pmod4$.}\end{array}\right.$ Similarly, we have
$(p-1,N_0)=1$ and $\frac{N}{(N,\frac{q-1}{p-1})}=1$. Then, by
Theorem\ref{r=2-4}(ii)
$$G(\chi)=\left\{\begin{array}{ll}(-1)^{\frac{p-1}{2}\frac{l_2-1}{2}+1}p^{\frac{f}{2}-h_1}(\frac{a+b\sqrt{-l_1}}{2})^4&\mbox{if $\lf{l_2}{l_1}=-1$ and $l_1{\equiv}3\;(\mbox{mod}\,8)$};\\
 (-1)^{\frac{p-1}{2}\frac{l_2-1}{2}+1}p^{\frac{f}{2}}&\mbox{otherwise,}  \end{array}\right.$$
where $h_1=h(\mathbb{Q}(\sqrt{-l_1}))$ and integers $a,b$ are
determined by equations (\ref{(a-1)}) in Theorem\ref{r=2-3}.

To summarize, we obtain,
 \begin{thm}\label{thm-e}
Let $N=2l_1^{r_1}l_2^{r_2}$, $[(\mathbb{Z}/N\mathbb{Z})^*:<{p}>]=2$,
i.e. $f=\frac{\varphi(N)}{2}$. Take $q=p^f$, $\chi$ a primitive the
multiplicative character of order $N$ over $\mathbb{F}_q$ and $h_1,\
h_{12}$ be respectively the ideal class numbers of
$\mathbb{Q}(\sqrt{-l_1})$ and $\mathbb{Q}(\sqrt{-l_1l_2})$. Assume
that the orders of $p$ in group $(\mathbb{Z}/l_1^{r_1}\mathbb{Z})^*$
and $(\mathbb{Z}/l_2^{r_2}\mathbb{Z})^*$ are respectively
$\varphi(l_1^{r_1})/a_0$ and $\varphi(l_2^{r_2})/a_1$. Then
 \par
(i).~for Case E1, ($a_1=a_2=1,\ (l_1,l_2){\equiv}(3,1)\pmod4$)
 $$G(\chi)=\left\{\begin{array}{ll}-p^{\frac{f}{2}}&\mbox{if $l_1l_2{\equiv}7\pmod8$;}\\-p^{\frac{f}{2}-h_{12}}(\frac{a'+b'\sqrt{-l_1l_2}}{2})^2&\mbox{if $l_1l_2{\equiv}3\pmod8$,}\end{array}\right.$$
where integers $a',b'$ are determined by equations (\ref{(b1)}) in
Theorem\ref{r=2-4}.
\par
For Case E2, ($a_1=2,\ a_2=1,\ 3\neq l_1{\equiv}3\pmod4$)
$$G(\chi)=\left\{\begin{array}{ll}(-1)^{\frac{p-1}{2}\frac{l_2-1}{2}+1}p^{\frac{f}{2}-h_1}(\frac{a+b\sqrt{-l_1}}{2})^4&\mbox{if $\lf{l_2}{l_1}=-1$ and $l_1{\equiv}3\;(\rm{mod}\,8)$;}\\
 (-1)^{\frac{p-1}{2}\frac{l_2-1}{2}+1}p^{\frac{f}{2}}&\mbox{otherwise,}  \end{array}\right.$$
where integers $a,b$ are determined by equations (\ref{(aa)}) in
Case A.
 \end{thm}

\begin{exmp}\ \ (For \underline{Case E1}).\quad
\par(1).\  Let $l_1=5,\ l_2=3,\ N=30,\ f=\frac{\varphi(30)}{2}=4,\
h_{12}=2$. The minimum primitive root modulo 5 $g_1=7$ such that
$g_1{\equiv}1\pmod 3$, while The minimum primitive root modulo 3
$g_2=11$ such that $g_2{\equiv}1\pmod 5$. Then
$p{\equiv}77{\equiv}2\pmod{15}.$

Take $p=17$, by Theorem\ref{thm-e}(i), we have that the Gauss sum
$G(\chi)$ of order 30 over $\mathbb{F}_{17^{4}}$ is
$$G(\chi)=-17^2.$$
\par(2).\ Let $l_1=5,\ l_2=7,\ N=70,\ f=\frac{\varphi(70)}{2}=12,\
h_{12}=2$. The minimum primitive root modulo 5 $g_1=8$ such that
$g_1{\equiv}1\pmod 7$, while The minimum primitive root modulo 7
$g_2=26$ such that $g_2{\equiv}1\pmod 5$. Then
$p{\equiv}33\pmod{35}.$

Take $p=103$, and $a=199,b=\pm9$ are solutions of equations
(\ref{(aa)}). Then, by Theorem\ref{thm-e}(ii), we have that the
Gauss sum $G(\chi)$ of order 70 over $\mathbb{F}_{103^{12}}$ and its
conjugation are
$$\{G(\chi),\overline{G(\chi)}\}=\left\{-103^4\left(\frac{199\pm9\sqrt{-35}}{2}\right)^2\right\}.$$\end{exmp}

\vskip2mm{\large \maltese \quad\underline{\textbf{Case
F}}.}\vskip2mm

In Case F, $N=4l_1^{r_1}=4N_0$. Let $g_0, g_1$ defined as in Section
1. So, we consider the following subcases according to the values of
$a_0,\ a_1$.

\vskip2mm\underline{\textbf{Case F1}}.\
$a_0=a_1=1$，$p{\equiv}g_0g_1\pmod{N}$, i.e. $p{\equiv}3\pmod4$ and
$p{\equiv}g_1\pmod{N_0}$, $l_1{\equiv}1\pmod4$,
$f=\frac{\varphi(N)}{2}=\varphi(N_0)=(l_1-1)l_1^{r_1-1}{\equiv}0\pmod4$.
$K=\mathbb{Q}(\sqrt{-l_1})$ (we note that
$K\not\subset\mathbb{Q}(\zeta_{N_0})$).

 \begin{lem}\label{lemd1}
In Case F1, the order of $\chi\big|_{\mathbb{F}_p}$ is 1, i.e.,
$\chi\big|_{\mathbb{F}_p}$ is trivial.
 \end{lem}
 \begin{proof}
We claim that $(p-1,N_0)=1$. Otherwise, $l_1\mid p-1\ \Rightarrow\
p{\equiv}1\pmod{l_1}\ \Rightarrow\
p^{l_1^{r_1-1}}=1\pmod{l_1^{r_1}=N_0}$. This is contradict to
$p{\equiv}g_1\pmod{l_1}$.

By Lemma \ref{lem0-2}, the order of $\chi\big|_{\mathbb{F}_p}$ is
$$\dfrac{N}{(N,\frac{q-1}{p-1})}=\dfrac{4l_1^{r_1}}{(4l_1^{r_1},\frac{4l_1^{r_1} n}{p-1})}=\dfrac{4}{(4,\frac{4n}{p-1})}.$$
Let $p=2k+1$ where $k\geqslant{}1$ is odd integer, for
$p{\equiv}3\pmod4$. Since $4|f$, we have that
$$\frac{q-1}{p-1}=\frac{p^{f}-1}{p-1}=\frac{(2k)^f+\binom{f}{f-1}(2k)^{f-1}+\cdots+\binom{f}{1}2k}{2k}{\equiv}\binom{f}{2}2k+f{\equiv}0\pmod4.$$
Then $(4,\frac{q-1}{p-1})=(4,\frac{4n}{p-1})=4$ and the lemma has
been proved.
 \end{proof}

Let $R_2$ and $\overline{R}_2$ respectively denote the quadratic
residue set and quadratic non-residue set in group
$(\mathbb{Z}/N\mathbb{Z})^*$, and take
$$\begin{array}{ll}R_2^{(1)}=\{x\in R_2\mid x{\equiv}1\pmod4\}\quad& R_2^{(3)}=\{x\in R_2\mid x{\equiv}3\pmod4\}\\
  \overline{R}_2^{(1)}=\{x\in \overline{R}_2\mid x{\equiv}1\pmod4\}\quad& \overline{R}_2^{(3)}=\{x\in \overline{R}_2\mid x{\equiv}3\pmod4\}.\end{array}$$
Then $<{p}>=R_2^{(1)}\cup \overline{R}_2^{(3)}$,
$g_0<{p}>=g_1<{p}>=R_2^{(3)}\cup \overline{R}_2^{(1)}$. We define a
isomorphic mapping between $(\mathbb{Z}/N\mathbb{Z})^*$ and
$(\mathbb{Z}/N_0\mathbb{Z})^*$ as
\begin{equation}\label{(d-2)}\begin{array}{crcl}\Phi:&(\mathbb{Z}/N\mathbb{Z})^*&\stackrel{\sim}{\longrightarrow}&(\mathbb{Z}/N_0\mathbb{Z})^*\\
 &x&\longmapsto&x\pmod{N_0}\\
 \Phi^{-1}:&y+\frac{1+(-1)^y}{2}&\longleftarrow\hspace{-0.7mm}\shortmid&y.\end{array}\end{equation}
Considering the group isomorphism
$$(\mathbb{Z}/N\mathbb{Z})^*\cong\{\pm1\}\times(\mathbb{Z}/2N_0\mathbb{Z})^*\cong(\mathbb{Z}/2N_0\mathbb{Z})^*\cup[(\mathbb{Z}/2N_0\mathbb{Z})^*+2N_0],$$
each element $s$ in $(\mathbb{Z}/N\mathbb{Z})^*$ can be viewed as
 \begin{eqnarray}\label{(f-1)}s&=&s_0+\frac{(-1)^{s_0}+1}{2}l_1^{r_1}+2l_1^{r_1}\cdot j\\
    &=&s_0+\left[\frac{(-1)^{s_0}+1}{2}+2j\right]l_1^{r_1}\qquad(\mbox{where\ }s_0\in(\mathbb{Z}/N_0\mathbb{Z})^*,j=0,1). \nonumber\end{eqnarray}
So, we find that, when $j=0$,
$(\mathbb{Z}/2N_0\mathbb{Z})^*=R_2^{(1)}\cup \overline{R}_2^{(3)}$,
and when $j=1$, $(\mathbb{Z}/2N_0\mathbb{Z})^*+2N_0=R_2^{(3)}\cup
\overline{R}_2^{(1)}$. Then
$$<{p}>=(\mathbb{Z}/2N_0\mathbb{Z})^*,\qquad g_0<{p}>=g_1<{p}>=(\mathbb{Z}/2N_0\mathbb{Z})^*+2N_0.$$
By Stickelberger's Theorem,
$$(G(\chi))O_K=\wp^{Ib_0+\sigma_{-1}b_1}\quad(\mbox{where\ ~}b_0=\frac{1}{N}\sum\limits_{s\in(\mathbb{Z}/N\mathbb{Z})^*\atop s\in<{p}>}s,\;\;b_1=\frac{1}{N}\sum\limits_{s\in(\mathbb{Z}/N\mathbb{Z})^*\atop s\in-<{p}>}s).$$
So,
$$\begin{array}{rl}b_0&=\frac{1}{N}\sum\limits_{s\in(\mathbb{Z}/N\mathbb{Z})^*\atop
s\in<{p}>}s=\frac{1}{N}\sum\limits_{s\in(\mathbb{Z}/2N_0\mathbb{Z})^*}s=\frac{1}{N}\sum\limits_{s_0\in(\mathbb{Z}/N_0\mathbb{Z})^*}(s_0+\frac{(-1)^{s_0}+1}{2}l_1^{r_1})\\[7mm]
  &=\frac{1}{N}\left[\sum\limits_{y=0}^{l_1^{r_1-1}-1}\sum\limits_{x=1}^{l_1-1}(x+l_1y)+\frac{f}{2}l_1^{r_1}\right]=\frac{1}{N}[\frac{l_1^{r_1}}{2}f+\frac{l_1^{r_1}}{2}f]=\frac{f}{4}\\[7mm]
  b_1&=\frac{1}{N}\sum\limits_{s\in(\mathbb{Z}/N\mathbb{Z})^*\atop s\in-<{p}>}s=\frac{1}{N}\sum\limits_{s\in(\mathbb{Z}/2N_0\mathbb{Z})^*}(s+2l_1^{r_1})=\frac{3}{4}f.\end{array}$$
Thus,
$(G(\chi))=\wp^{\frac{1}{4}}\overline{\wp}^{\frac{3}{4}f}=p^{\frac{f}{4}}\overline{\wp}^{\frac{f}{2}}$.
Ideal $\overline{\wp}^{\frac{f}{2}}$ is principal ideal and we
assume that $\overline{\wp}^{\frac{f}{2}}=(a+b\sqrt{-l_1})$, where
$a,b\in\mathbb{Z}$ are determined by $a^2+b^2l_1=p^{\frac{f}{2}},\
p\nmid b$. Then
$$G(\chi)=\varepsilon p^{\frac{f}{4}}(a+b\sqrt{-l_1}).$$
As we known, $\varepsilon$ is a unit root in $K$. Since
$l_1\neq1,3$, $\varepsilon=\pm1$. It means that we need to determine
the sign of integer $a$. Considering the $l_1^{r_1}$ power of
$G(\chi)$, if we let $\varepsilon=1$, then
$$G^{l_1^{r_1}}(\chi){\equiv}\left\{\begin{array}{l}\sum\limits_{x}\chi^{l_1^{r_1}}(x)\zeta_p^{l_1^{r_1} T(x)}{\equiv}\overline{\chi}^{l_1^{r_1}}(l_1^{r_1})G(\chi^{l_1^{r_1}}){\equiv}-p^{\frac{f}{2}}\pmod{l_1}\\
  ap^{\frac{f}{4}}\pmod{\sqrt{-l_1}}.\end{array}\right.$$
Therefore,
$$a{\equiv}-p^{\frac{f}{4}}\pmod{l_1}.$$

\vskip2mm \underline{\textbf{Case F2}}.\quad
$a_0=2,\;a_1=1\;p{\equiv}g_0^2g_1\pmod{N}$, i.e. $p{\equiv}1\pmod4$
and $p{\equiv}g_1\pmod{N_0}$, $K=\mathbb{Q}(\sqrt{-1})$.
$f=\varphi({N})/2=\varphi(N_0){\equiv}\left\{\begin{array}{ll}2\pmod4&\mbox{if
$l_1{\equiv}3\pmod4$;}\\0\pmod4&\mbox{if
$l_1{\equiv}1\pmod4$.}\end{array}\right.$ Similar as Lemma
\ref{lemd1}, we have $(p-1,N_0)=1,\;(p-1,4)=4$ and the order of
$\chi\big|_{\mathbb{F}_p}$ is
$$\frac{N}{(N,\frac{q-1}{p-1})}=\left\{\begin{array}{ll}2\pmod4&\mbox{if $l_1{\equiv}3\pmod4$;}\\1\pmod4&\mbox{if $l_1{\equiv}1\pmod4$.}\end{array}\right.$$

\underline{If $l_1{\equiv}1\pmod4$}, $G(\chi)\in O_K$, by
Stickelberger's Theorem, we have
$$(G(\chi))O_K=\wp^{Ib_0+\sigma_{-1}b_1},$$
where
$$\begin{array}{rl}b_0&=\frac{1}{N}\sum\limits_{s\in(\mathbb{Z}/N\mathbb{Z})^*\atop s\in<{p}>}s=\frac{1}{N}\sum\limits_{s{\equiv}1({\,\rm mod\,}4)}s,\\
         b_1&=\frac{1}{N}\sum\limits_{s\in(\mathbb{Z}/N\mathbb{Z})^*\atop s\in-<{p}>}s=\frac{1}{N}\sum\limits_{s{\equiv}3({\,\rm mod\,}4)}s.\end{array}$$
As we known, $b_0+b_1=f$. Next, we calculate $b_0-b_1$. Similar as
Case F1, we consider
$$\xymatrix{(\mathbb{Z}/N_0\mathbb{Z})^*\ar[r]_{1:1}^{\Phi_1}&(\mathbb{Z}/2N_0\mathbb{Z})^*\ar[r]_{1:2}^{\Phi_2}&(\mathbb{Z}/N\mathbb{Z})^*},$$
where $\Phi_1,\ \Phi_2$  are, respectively, defined by (\ref{(d-2)})
and (\ref{(f-1)}). For the pairs $(s_0,l_1^{r_1}-s_0)\ \
(s_0\in(\mathbb{Z}/N_0\mathbb{Z})^*)$, we find that
$\Phi_1(s_0)+\Phi_1(l_1^{r_1}-s_0)$ are always equal to
$l_1+l_1^{r_1}$, and
$\Phi_1(s_0)-\Phi_1(l_1^{r_1}-s_0){\equiv}0\pmod4$. Since
$\#\{s\in(\mathbb{Z}/2N_0\mathbb{Z})^*\mid
s{\equiv}1\pmod4\}=\#\{s\in(\mathbb{Z}/2N_0\mathbb{Z})^*\mid
s{\equiv}3\pmod4\}$,
$$b_0-b_1=\sum\limits_{s\in(\mathbb{Z}/N\mathbb{Z})^*}s\lf{-1}{s}=0.$$
Therefore
$$G(\chi)=\varepsilon p^{\frac{f}{2}},$$
where $\varepsilon$ is a unit root in $O_K$, i.e.
$\varepsilon\in\{\pm1,\pm{i}\}\ \ ({i}=\sqrt{-1})$, which can be
determined by Stickelberger congruence \cite[\S11.3]{B-E-W}. (The
detailed discussing of how to determine the sign or unit root
ambiguities for Gauss sums is given by \cite{Y4}.)

\underline{If $l_1{\equiv}3\pmod4$},
$G(\chi)=\left(\sum\limits_{T(x)=1}\chi(x)\right)G(\chi\big|_{\mathbb{F}_p})=\left(\sum\limits_{T(x)=1}\chi(x)\right)\sqrt{p}\in
O_K[\zeta_p]$. By Stickelberger's Theorem, we have
$$(G(\chi))O_M=\mathfrak{P}^{b_0I+b_1\sigma_{-1}}\quad(\mbox{where~}b_0=\frac{p-1}{N}\sum\limits_{s\in(\mathbb{Z}/N\mathbb{Z})^*\atop s{\equiv}1{\rm\,(mod\,}4)}s,\ b_1=\frac{p-1}{N}\sum\limits_{s\in(\mathbb{Z}/N\mathbb{Z})^*\atop s{\equiv}3{\rm\,(mod\,}4)}s).$$
Take $d_1:=\#\{s\in(\mathbb{Z}/2N_0\mathbb{Z})^*\mid
s{\equiv}1\pmod4\},\ d_3:=\#\{s\in(\mathbb{Z}/2N_0\mathbb{Z})^*\mid
s{\equiv}3\pmod4\}$, Similarly by the group maps $\Phi_1,\ \Phi_2$,
$$\begin{array}{rl}
 b_0&=\frac{p-1}{N}[\sum\limits_{s\in(\mathbb{Z}/2N_0\mathbb{Z})^*\atop s{\equiv}1{\rm\,(mod\,}4)}s+\sum\limits_{s\in(\mathbb{Z}/2N_0\mathbb{Z})^*\atop
 s{\equiv}3{\rm\,(mod\,}4)}(s+2l_1^{r_1})]=\frac{p-1}{N}[l_1^{r_1} f+2l_1^{r_1} d_3]\\
 b_1&=\frac{p-1}{N}[\sum\limits_{s\in(\mathbb{Z}/2N_0\mathbb{Z})^*\atop s{\equiv}3{\rm\,(mod\,}4)}s+\sum\limits_{s\in(\mathbb{Z}/2N_0\mathbb{Z})^*\atop
 s{\equiv}1{\rm\,(mod\,}4)}(s+2l_1^{r_1})]=\frac{p-1}{N}[l_1^{r_1} f+2l_1^{r_1} d_1]
 \end{array}$$

\begin{lem}\label{lemf2} In Case F2, if $l_1{\equiv}3\pmod4$, then
 $$d_1=\frac{\varphi(l_1^{r_1})}{2}+1,\quad d_3=\frac{\varphi(l_1^{r_1})}{2}-1,$$
i.e. there are $\frac{\varphi(l_1^{r_1})}{2}+1$ elements $s$ in
$(\mathbb{Z}/2N_0\mathbb{Z})^*$ such that $s{\equiv}1\pmod4$, and
$\frac{\varphi(l_1^{r_1})}{2}-1$ elements $s$ in
$(\mathbb{Z}/2N_0\mathbb{Z})^*$ such that $s{\equiv}3\pmod4$.
\end{lem}
 \begin{proof}
 $\#(\mathbb{Z}/2N_0\mathbb{Z})^*=\varphi(l_1^{r_1})=(l_1-1)l_1^{r_1-1}$. Suppose that
$$S_1=\{\mbox{all the odd numbers from 1 to $2l_1^{r_1}-1$}\}.$$
Since $2l_1^{r_1}-1{\equiv}2(-1)^{r_1}-1{\equiv}1\pmod4$, there are
$\frac{l_1^{r_1}+1}{2}$ elements $s$ in $S_1$ such that
$s{\equiv}1\pmod4$, and $\frac{l_1^{r_1}-1}{2}$ elements $s$ such
that $s{\equiv}3\pmod4$.

Next, we consider  set $S_2=\{s\in S_1:l_1|s\}$. $\forall x\in S_2$,
$x$ must be the form as $x=l_1(2k-1)$, where
$k=1,\cdots,l_1^{r_1-1}$. Since
 $$x{\equiv}(-1)(2k-1){\equiv}2k+1{\equiv}\left\{\begin{array}{ll}3\pmod4,&\mbox{if $~k$\ is odd;}\\1\pmod4&\mbox{if $~k$ is even.}\end{array}\right.$$
Then there are $\frac{l_1^{r_1-1}-1}{2}$ elements $s$ in $S_2$ such
that $s{\equiv}1\pmod4$, and $\frac{l_1^{r_1-1}+1}{2}$ elements $s$
such that $s{\equiv}3\pmod4$.

Finally, since $(\mathbb{Z}/2N_0\mathbb{Z})^*=S_1\big\backslash
S_2$, the lemma has been proved.
 \end{proof}

By Lemma\ref{lemf2},
$$\left\{\begin{array}{rl}b_0+b_1&=\frac{p-1}{N}(2l_1^{r_1} f+2l_1^{r_1}(d_1+d_3))=(p-1)f\\b_0-b_1&=\frac{p-1}{N}2l_1^{r_1}(d_3-d_1)=-(p-1)\end{array}\right.~\Rightarrow~\left\{\begin{array}{l}b_0=\frac{f-1}{2}(p-1)\\b_1=\frac{f+1}{2}(p-1)\end{array}\right.$$
Then
$$(\frac{G(\chi)}{\sqrt{p}})=p^{\frac{f-2}{2}}\overline{\mathfrak{P}}^{p-1}=p^{\frac{f-2}{2}}\overline{\wp}\in O_K.$$
Assume that $\wp=(a+b\sqrt{-1})$, where $a,b\in\mathbb{Z}$ such that
$a^2+b^2=p,\ p\nmid b$. Thus,
$$G(\chi)=\varepsilon \sqrt{p}p^{\frac{f}{2}-1}(a+b\sqrt{-1}),$$
where $\varepsilon\in\{\pm1,\pm i\}$ can  be determined by
Stickelberger congruence. For more details, one can refer to
\cite{Y4}.

\vskip2mm \underline{\textbf{Case F3}}.\quad $a_0=1,\ a_1=2$,
$p{\equiv}g_0g_1^2\pmod{N}$, i.e. $p{\equiv}3\pmod4$ and
$p{\equiv}g_1^2\pmod{N_0}$, $3\neq l_1{\equiv}3\pmod4$,
$K=\mathbb{Q}(\sqrt{-l_1})$.

Here $f=\varphi(N)/2=\varphi(N_0)=(l_1-1)l_1^{r_1-1}{\equiv}2\pmod4$
and $(p-1,l_1^{r_1})=1$, then the order of
$\chi\big|_{\mathbb{F}_p}$ is $\frac{N}{(N,\frac{q-1}{p-1})}=1$. By
Stickelberger's Theorem, $(G(\chi))O_K=\wp^{b_0I+b_1\sigma_{-1}}$,
where
$$\begin{array}{rl}
 b_0&=\frac{1}{N}\sum\limits_{s\in(\mathbb{Z}/N\mathbb{Z})^*\atop s\in<p>}s=\frac{1}{N}[2\sum\limits_{s_0\in(\mathbb{Z}/N_0\mathbb{Z})^*\atop\lf{s_0}{l_1}=1}(s_0+\frac{(-1)^{s_0}+1}{2}l_1^{r_1})+2l_1^{r_1}\frac{\varphi(l_1^{r_1})}{2}],\\
 b_1&=\frac{1}{N}\sum\limits_{s\in(\mathbb{Z}/N\mathbb{Z})^*\atop s\in-<p>}s=\frac{1}{N}[2\sum\limits_{s_0\in(\mathbb{Z}/N_0\mathbb{Z})^*\atop\lf{s_0}{l_1}=-1}(s_0+\frac{(-1)^{s_0}+1}{2}l_1^{r_1})+2l_1^{r_1}\frac{\varphi(l_1^{r_1})}{2}]. \end{array}$$
Then
 $$\begin{array}{rl}b_0+b_1&=\frac{1}{N}4l_1^{r_1}\varphi(l_1^{r_1})=f,\\[5mm]
  b_0-b_1&=\frac{1}{N}\left[-2h_1l_1^{r_1}+l_1^{r_1}\sum\limits_{s_0\in(\mathbb{Z}/N_0\mathbb{Z})^*}(-1)^{s_0}\lf{s_0}{l_1}\right]\\
  &=\left\{\begin{array}{ll}-\frac{h_1}{2}+\frac{1}{4}2h_1=0&\mbox{if~}l_1{\equiv}7\pmod8;\\-\frac{h_1}{2}-\frac{1}{4}6h_1=-2h_1&\mbox{if~}l_1{\equiv}3\pmod8,\end{array}\right. \end{array}$$
where the last equation, above, is by the formula on the class
number of imaginary quadratic field $\mathbb{Q}(\sqrt{-l_1})$.

\underline{If $l_1{\equiv}7\pmod8$}, $G(\chi)=\varepsilon
p^{\frac{f}{2}}$, where $\varepsilon\in\{\pm1\}$. We can determine
$\varepsilon$ by the congruences below
$$G^{l_1^{r_1}}(\chi){\equiv}\left\{\begin{array}{l}\overline{\chi}^{l_1^{r_1}}(l_1^{r_1})G(\chi^{l_1^{r_1}}){\equiv}(-1)^{\frac{p+1}{4}}p^{\frac{f}{2}}\pmod{l_1}\\
  \varepsilon p^{\frac{f}{2}}\pmod{{-l_1}}.\end{array}\right.$$
Hence, $\varepsilon=(-1)^{\frac{p+1}{4}}$,
$G(\chi)=(-1)^{\frac{p+1}{4}}p^{\frac{f}{2}}$.

\underline{If $l_1{\equiv}3\pmod8$}, we have $G(\chi)=\varepsilon
p^{\frac{f}{2}-h_1}(\frac{a+b\sqrt{-l_1}}{2})^2$ by the result in
Case A, where integers $a,b$ are determined by equations
(\ref{(aa)}). And $\varepsilon$ can be determined by the congruences
below
$$G^{l_1^{r_1}}(\chi){\equiv}\left\{\begin{array}{l}\overline{\chi}^{l_1^{r_1}}(l_1^{r_1})G(\chi^{l_1^{r_1}}){\equiv}(-1)^{\frac{p+1}{4}}p^{\frac{f}{2}}\pmod{l_1}\\
  \varepsilon p^{\frac{f}{2}-h_1}\frac{a^2}{4}\pmod{\sqrt{-l_1}}\end{array}\right.$$
Hence, $\varepsilon{\equiv}(-1)^{\frac{p+1}{4}}\pmod{l_1}$.

\vskip2mm In conclusion, we obtain that
 \begin{thm}\label{thm-f}
Let $N=4l_1^{r_1}$, $[(\mathbb{Z}/N\mathbb{Z})^*:<{p}>]=2$, i.e.
$f=\frac{\varphi(N)}{2}$. Take $q=p^f$, $\chi$ be a multiplicative
character of order $N$ over $\mathbb{F}_q$ and $h_1$ be the ideal
class number of $\mathbb{Q}(\sqrt{-l_1})$. Assume that the orders of
$p$ in group $(\mathbb{Z}/4\mathbb{Z})^*$ and
$(\mathbb{Z}/l_1^{r_1}\mathbb{Z})^*$ are respectively $2/a_0$ and
$\varphi(l_1^{r_1})/a_1$. Then
 \par
(i).~For Case F1, ($a_0=a_1=1,\ p{\equiv}3\pmod4,\
l_1{\equiv}1\pmod4$)
$$G(\chi)=p^{\frac{f}{4}}(a'+b'\sqrt{-l_1}),$$
where integers $a',b'$ are determined by equations
\begin{equation}\label{(f-3)}\left\{\begin{array}{l}(a')^2+l_1(b')^2=p^{\frac{f}{2}};\\a'{\equiv}-p^{\frac{f}{4}}\pmod{l_1}.\end{array}\right.\end{equation}
\par
(ii).~For Case F2, (($a_0=2,\ a_1=1,\ p{\equiv}1\pmod4$))
$$G(\chi)=\left\{\begin{array}{ll}\varepsilon p^{\frac{f}{2}}&\mbox{if $l_1{\equiv}1\pmod4$;}\\
 \varepsilon \sqrt{p}p^{\frac{f}{2}-1}(a''+b''\sqrt{-1})&\mbox{if $l_1{\equiv}3\pmod4$,}  \end{array}\right.$$
where $\varepsilon\in\{\pm1,\pm{i}\}$ can be determined by
Stickelberger congruence and integers $|a''|,|b''|$ are determined
by $(a'')^2+(b'')^2=p$.
\par
(iii).~For Case F3, ($a_0=1,\ a_1=2,\ p{\equiv}3\pmod4,\
l_1{\equiv}3\pmod4$)
$$G(\chi)=\left\{\begin{array}{ll}(-1)^{\frac{p+1}{4}}p^{\frac{f}{2}}&\mbox{if $l_1{\equiv}7\pmod8$;}\\
 (-1)^{\frac{p+1}{4}}p^{\frac{f}{2}-h_1}(a+b\frac{-1+\sqrt{-l_1}}{2})&\mbox{if $l_1{\equiv}3\pmod8$,}  \end{array}\right.$$
where integers $a,b$ are determined by equations (\ref{(aa)}).
 \end{thm}

%因为在计算情形F下Gauss和的过程中，使用到Stickelberger同余式，所以关于情形F的具体数值实例，我们将在下一章中进一步讨论Stickelberger同余式后给出。
\vskip1cm

\subsection{Explicit evaluation of Gauss sums $G(\chi^\lambda)\quad(1\leqslant{}\lambda<N)$}

In this section, we give explicit evaluation of a series Gauss sums
$G(\chi^\lambda)\quad(1\leqslant{}\lambda<N)$, by the results given
in section 4.1. Our main mathematical tool is Davenport-Hasse
lifting formula, and we omit some similar proofs for simplicity.

\vskip2mm{\large \maltese \quad\underline{\textbf{Case
A}}.}\vskip2mm

$N=l_1^{r_1},~3\neq l_1{\equiv}3\pmod4$. For
$1\leqslant{}\lambda\leqslant{}N-1$, let $\lambda=xl_1^t+y$ for some
integer $t<r_1$, where $0\leqslant{}x\leqslant{}l_1-1,\
0\leqslant{}y\leqslant{}l_1-1$ and $(x,y)\neq(0,0)$. If $y\neq0$,
$(\lambda,l_1^{r_1})=1$, then
$$G(\chi^\lambda)=\left\{\begin{array}{ll}G(\chi)&\mbox{if $y\in<{p}>\subset(\mathbb{Z}/l_1\mathbb{Z})^*$;}\\ G(\overline{\chi})=\overline{G(\chi)}&\mbox{if $y\not\in<{p}>\subset(\mathbb{Z}/l_1\mathbb{Z})^*$.}\end{array}\right.$$
If $y=0$, $\lambda=xl^t$, then
$$G(\chi^\lambda)=\left\{\begin{array}{ll}G(\chi^{l^t})&\mbox{if $x\in<{p}>\subset(\mathbb{Z}/l_1\mathbb{Z})^*$;}\\\overline{G(\chi^{l^t})}&\mbox{if $x\not\in<{p}>\subset(\mathbb{Z}/l_1\mathbb{Z})^*$.}\end{array}\right.$$
Thus, the problem owns to the calculation of Gauss sums
$G(\chi^{l^{t}})\ \ (0\leqslant{}t\leqslant{}r_1-1)$.

The order of $\chi^{l^{t}}$ is $\widetilde{N}=l_1^{r_1-t}$ and ${\rm
ord}_{\widetilde{N}}(p)=\frac{\varphi(\widetilde{N})}{2}=\widetilde{f}$.
Let $\eta$ be the corresponding primitive character of order
$\widetilde{N}$ over $\mathbb{F}_{p^{\widetilde{f}}}$, then by
Theorem\ref{r=2-3},
$$G(\eta)=p^{\frac{\widetilde{f}-h_1}{2}}({a}+{b}\sqrt{-l_1})/2,$$
where integer ${a},{b}$ are determined by equations (\ref{(aa)}).
Then by Davenport-Hasse lifting formula, the theorem is given as
follow .

 \begin{thm}[Case A,\cite{La}]\label{thm-a2}
Let $N=l_1^{r_1},~l_1{\equiv}3\pmod4,~l_1\geqslant{}7$,
$[(\mathbb{Z}/N\mathbb{Z})^*:<{p}>]=2$, i.e ${\rm
ord}_N(p)=f=\frac{1}{2}{\varphi(N)}$. Take $q=p^f$, $\chi$ be a
multiplicative character of order $N$ over $\mathbb{F}_q$ and
$h_1=h(\mathbb{Q}(\sqrt{-l_1}))$. Then, for $0\leqslant{}t<r_1$,
$$G(\chi^{l_1^{t}})=p^{\frac{f-h_1l_1^{t}}{2}}\left(\frac{a+b\sqrt{-l_1}}{2}\right)^{l_1^t},$$
where integers $a,b$ are determined by (\ref{(aa)}).
 \end{thm}

\vskip2mm{\large \maltese \quad\underline{\textbf{Case
B}}.}\vskip2mm

Here, $N=l_1^{r_1}l_2^{r_2}$, where $l_1,l_2$ are distinct odd
primes and $r_1,r_2\geqslant{}1$. Similarly as Case A, we only
consider the Gauss sums $G(\chi^{l_1^{t_1}l_2^{t_2}}),\
G(\chi^{l_1^{r_1}l_2^{t_2}})$ and $G(\chi^{l_1^{t_1}l_2^{r_2}})\ \
(0\leqslant{}t_1<r_1,0\leqslant{}t_2<r_2$). And the result is given
as follow.

 \begin{thm}[Case B]\label{thm-b2}
Let $N=l_1^{r_1}l_2^{r_2}$, $[(\mathbb{Z}/N\mathbb{Z})^*:<{p}>]=2$,
i.e. ${\rm ord}_N(p)=f=\frac{1}{2}{\varphi(N)}$. Take $q=p^f$,
$\chi$ be a multiplicative character of order $N$ over
$\mathbb{F}_q$. Then,
\par (1). ~For \underline{Case B1}\ ($0\leqslant{}t_1<r_1,\
0\leqslant{}t_2<r_2$),
$$\begin{array}{l}
  G(\chi^{l_1^{t_1}l_2^{t_2}})=p^{\frac{1}{2}(f-h_{12}l_1^{t_1}l_2^{t_2})}\left(\frac{a'+b'\sqrt{-l_1l_2}}{2}\right)^{l_1^{t_1}l_2^{t_2}},\\
  G(\chi^{l_1^{r_1}l_2^{t_2}})=p^{\frac{f}{2}},\\
  G(\chi^{l_1^{t_1}l_2^{r_2}})=-p^{\frac{f}{2}},\\
 \end{array}$$
where $h_{12}=h(\mathbb{Q}(\sqrt{-l_1l_2}))$ and integers $a',b'$
are determined by equations (\ref{(b1)}).
\par (2). ~For \underline{Case B2}\ ($0\leqslant{}t_1<r_1,\
0\leqslant{}t_2<r_2$),
$$\begin{array}{l}
  G(\chi^{l_1^{t_1}l_2^{t_2}})=\left\{\begin{array}{lll} p^{\frac{f}{2}},&\mbox{if}\ \left(\frac{l_2}{l_1}\right)=1;\\varphi^{\frac{f}{2}-h_1l_1^{t_1}l_2^{t_2}}(\frac{a+b\sqrt{-l_1}}{2})^{2l_1^{t_1}l_2^{t_2}},&\mbox{if}\ \left(\frac{l_2}{l_1}\right)=-1,\end{array}\right.\\
  G(\chi^{l_1^{r_1}l_2^{t_2}})=p^{\frac{f}{2}},\\
  G(\chi^{l_1^{t_1}l_2^{r_2}})=-p^{\frac{1}{2}(f-h_1l_1^{t_1}\varphi(l_2^{r_2}))}\left(\frac{a+b\sqrt{-l_1}}{2}\right)^{l_1^{t_1}\varphi(l_2^{r_2})},
 \end{array}$$
where $h_1=h(\mathbb{Q}(\sqrt{-l_1}))$ and integers $a,b$ are
determined by equations (\ref{(aa)}).
\end{thm}

\begin{proof}
Let $\chi^\lambda=\chi^{l_1^xl_2^y}\ \
(0\leqslant{}x\leqslant{}r_1,0\leqslant{}y\leqslant{}r_2)$, which
correspond the primitive characters $\eta$ of order $\widetilde{N}$
over $\mathbb{F}_{p^{\widetilde{f}}}$ with $\widetilde{f}={\rm
ord}_{\widetilde{N}}(p)$. We list all the subcases in the following
tables. Then, the theorem can be proved by Davenport-Hasse lifting
formula.

(1). ~For Case B1,
  \begin{center}{\small{\begin{tabular}{ccccc}\hline\\[-3mm]
 $\lambda$&$\widetilde{N}$&$\widetilde{f}$&$f/\widetilde{f}$&Subcases of $G(\eta)$\\ \hline\\[-3mm]
 $l_1^{t_1}l_2^{t_2}$&$l_1^{r_1-t_1} l_2^{r_2-t_2}$&$\frac{1}{2}\varphi(l_1^{r_1-t_1})\varphi(l_2^{r_2-t_2})$&$l_1^{t_1}l_2^{t_2}$&Case B1\\[1mm]
 $l_1^{r_1}l_2^{t_2}$&$l_2^{r_2-t_2}$&$\varphi(l_2^{r_2-t_2})$&$\frac{\varphi(l_1^{r_1})}{2}l_1^{t_1}$&Pure Gauss sum\\[1mm]
 $l_1^{t_1}l_2^{r_2}$&$l_1^{r_1-t_1}$&${\varphi(l_1^{r_1-t_1})}$&$l_1^{t_1}\frac{\varphi(l_2^{r_2})}{2}$&Pure Gauss sum\\[2mm]\hline
 \end{tabular}}}\end{center}

(2). ~For Case B2,
  \begin{center}{\small{\begin{tabular}{ccccc}\hline\\[-3mm]
 $\lambda$&$\widetilde{N}$&$\widetilde{f}$&$f/\widetilde{f}$&Subcases of $G(\eta)$\\ \hline\\[-3mm]
 $l_1^{t_1}l_2^{t_2}$&$l_1^{r_1-t_1} l_2^{r_2-t_2}$&$\frac{1}{2}\varphi(l_1^{r_1-t_1})\varphi(l_2^{r_2-t_2})$&$l_1^{t_1}l_2^{t_2}$&Case B2\\[1mm]
 $l_1^{r_1}l_2^{t_2}$&$l_2^{r_2-t_2}$&$\varphi(l_2^{r_2-t_2})$&$\frac{\varphi(l_1^{r_1})}{2}l_1^{t_1}$&Pure Gauss sum\\[1mm]
 $l_1^{t_1}l_2^{r_2}$&$l_1^{r_1-t_1}$&$\frac{1}{2}{\varphi(l_1^{r_1-t_1})}$&$l_1^{t_1}{\varphi(l_2^{r_2})}$&Case A\\[2mm]\hline
 \end{tabular}}}\end{center}
\end{proof}

\vskip2mm{\large \maltese \quad\underline{\textbf{Case
C}}.}\vskip2mm

For $N=2^t\ \ (t\geqslant{}3)$, the following results were given by
\cite{M-V}.

 \begin{thm}\label{thm-c2}(\cite{M-V})
Let $N=2^t\ (t\geqslant{}3)$, $\chi$ be a multiplicative character
of order $N$ over $\mathbb{F}_q$. Then, for
$1\leqslant{}s\leqslant{}t$, Gauss sums $G(\chi^{2^{t-s}})$ are
given as follow:
 \par
 (i).~If $p{\equiv}3\pmod8$, $K=\mathbb{Q}(\sqrt{-2})$, then
$$\begin{array}{ll}G(\chi)=\sqrt{-p}p^{2^{t-3}-1}(a+b\sqrt{-2})&\mbox{for $s=t$};\\
  G(\chi^{2^{t-s}})=-(\sqrt{-1})^{2^{t-s}}\cdot p^{2^{t-3}-2^{t-s-1}}\cdot(a+b\sqrt{-2})^{2^{t-s}}&\mbox{for $3\leqslant{}s\leqslant{}t-1$};\\
  G(\chi^{2^{t-2}})=-p^{2^{t-3}}&\mbox{for $s=2$};\\
  G(\chi^{2^{t-1}})=\left\{\begin{array}{ll}p&\mbox{if $t=3$};\\-p^{2^{t-3}}&\mbox{if $t\geqslant{}4$}\end{array}\right.&\mbox{for
  $s=1$},
\end{array}$$ where $a,b\in\mathbb{Z}$ are given in Theorem\ref{r=2-5}(i).
 \par
(ii).~If $p{\equiv}5\pmod8$, $K=\mathbb{Q}({i})$, then
$$\begin{array}{ll}G(\chi)=p^{2^{t-3}}\sqrt{a+b{i}}\big{/}\sqrt[4]{p}&\mbox{for $s=t$};\\
  G(\chi^{2})=p^{2^{t-3}-1}\sqrt{p}(a+b{i})&\mbox{for $s=t-1$};\\
  G(\chi^{2^{t-s}})=-p^{2^{t-3}}(a+b{i})^{2^{t-s-1}}/p^{2^{t-s-2}}&\mbox{for $2\leqslant{}s\leqslant{}t-1$};\\
  G(\chi^{2^{t-1}})=-p^{2^{t-3}}&\mbox{for $s=1$};
\end{array}$$ where $a,b\in\mathbb{Z}$ and $\sqrt{a+b{i}}$ are given in
Theorem\ref{r=2-5}(2). \end{thm}

\vskip2mm{\large \maltese \quad\underline{\textbf{Case
D}}.}\vskip2mm

$N=2N_0=2l_1^{r_1}$, for the Gauss sums $G(\chi^{2^il_1^{t_1}})\ \
(i=0,1;\ 0\leqslant{}t_1\leqslant{}r_1,(i,t_1)\neq(1,r_1))$, we have
the following results:
 \begin{thm}[Case D]\label{thm-d2}
Let $N=2N_0=2l_1^{r_1},~3\neq l_1{\equiv}3\pmod4$,
$[(\mathbb{Z}/N\mathbb{Z})^*:<{p}>]=2$, i.e.
$f=\frac{\varphi(N)}{2}$. Take $q=p^f$ and $\chi$ be a
multiplicative characer of order $N$ over $\mathbb{F}_q$. Then, for
$0\leqslant{}t_1\leqslant{}r_1-1$,
$$\begin{array}{ll}
 G(\chi^{l_1^{t_1}})=\left\{\begin{array}{ll}(-1)^{\frac{p-1}{2}(r_1-t_1-1)+t_1}p^{\frac{f-1}{2}-h_1l_1^{t_1}}\sqrt{p^*}(\frac{a+b\sqrt{-l_1}}{2})^{2l_1^{t_1}}&\mbox{if $l_1{\equiv}3\pmod8$;}\\
         (-1)^{\frac{p-1}{2}(r_1-t_1)+t_1}p^{\frac{f-1}{2}}\sqrt{p^*}&\mbox{if $l_1{\equiv}7\pmod8$;}\end{array}\right.\\
 G(\chi^{2l_1^{t_1}})=p^{\frac{1}{2}(f-l_1^{t_1} h_1)}(a+b\sqrt{-l_1});\\
 G(\chi^{l_1^{r_1}})=(-1)^{f-1}(\sqrt{p^*})^{f}=(-1)^{\frac{p-1}{2}\frac{f-1}{2}}p^{\frac{f-1}{2}}\sqrt{p^*},
 \end{array}$$
where $h_1=h(\mathbb{Q}(\sqrt{-l_1}))$ and integers $a,b$ are
determined by equations (\ref{(aa)}). \end{thm}

\vskip2mm{\large \maltese \quad\underline{\textbf{Case
E}}.}\vskip2mm

$N=2N_0=2l_1^{r_1}l_2^{r_2}$. For the Gauss sums
$G(\chi^{2^il_1^{x}l_2^y})\ \ (i=0,1;\ 0\leqslant{}x\leqslant{}r_1,\
0\leqslant{}y\leqslant{}r_1,\ (i,x,y)\neq(1,r_1,r_2))$, we have the
following evaluations.

 \begin{thm}
Let $N=2N_0=2l_1^{r_1}l_2^{r_2}$,
$[(\mathbb{Z}/N\mathbb{Z})^*:<{p}>]=2$, i.e.
$f=\frac{\varphi(N)}{2}$. Take $q=p^f$ and $\chi$ be a
multiplicative character of order $N$ over $\mathbb{F}_q$. Assume
that the orders of $p$ in groups
$(\mathbb{Z}/l_1^{r_1}\mathbb{Z})^*$ and
$(\mathbb{Z}/l_2^{r_2}\mathbb{Z})^*$ are respectively
$\varphi(l_1^{r_1})/a_1$ and $\varphi(l_2^{r_2})/a_2$. Then, for
$0\leqslant{}t_1<r_1,\ 0\leqslant{}t_2<r_2$,
\par \underline{Case E1}: ($a_1=a_2=1$, $(l_1,l_2){\equiv}(3,1)\pmod4$)
$$\begin{array}{l}
 G(\chi^{l_1^{t_1}l_2^{t_2}})=\left\{\begin{array}{ll}-p^{\frac{f}{2}},&\mbox{if $l_1l_2{\equiv}3\pmod8$;}\\
         -p^{\frac{f}{2}-l_1^{t_1}l_2^{t_2} h_{12}}\left(\frac{a'+b'\sqrt{-l_1l_2}}{2}\right)^{2l_1^{t_1}l_2^{t_2}},&\mbox{if $l_1l_2{\equiv}7\pmod8$,}\end{array}\right.\\
 G(\chi^{l_1^{r_1}l_2^{t_2}})=(-1)^{\frac{p^{\frac{1}{2}\varphi(l_1^{r_1-t_1})}+1}{2l_1^{r_1-t_1}}}p^{\frac{f}{2}},\\
 G(\chi^{l_1^{t_1}l_2^{r_2}})= G(\chi^{l_1^{r_1}l_2^{r_2}})=G(\chi^{2l_1^{t_1}l_2^{r_2}})=-p^{\frac{f}{2}},\\
 G(\chi^{2l_1^{t_1}l_2^{t_2}})=p^{\frac{1}{2}(f-l_1^{t_1}l_2^{t_2} h_{12})}(\frac{a'+b'\sqrt{-l_1l_2}}{2})^{l_1^{t_1}l_2^{t_2}},\\
 G(\chi^{2l_1^{r_1}l_2^{t_2}})=p^{\frac{f}{2}},
 \end{array}$$
where $h_{12}$ is the ideal class number of
$\mathbb{Q}(\sqrt{-l_1l_2})$ and integers $a',b'$ are determined by
equations (\ref{(b1)}).
\par
\underline{Case E2}: ($a_1=2,\ a_2=1$, $3\neq l_1{\equiv}3\pmod4,\
l_2{\equiv}1\pmod2$)
$$\begin{array}{l}
 G(\chi^{l_1^{t_1}l_2^{t_2}})=\left\{\begin{array}{ll}(-1)^{\frac{p-1}{2}\frac{l_2-1}{2}}p^{\frac{f}{2}-2l_1^{t_1}l_2^{t_2} h_1}\left(\frac{a+b\sqrt{-l_1}}{2}\right)^{4l_1^{t_1}l_2^{t_2}},&\mbox{when $\lf{l_2}{l_1}$ and $l_1{\equiv}3\;(\rm{mod}\,8)$;}\\
         (-1)^{\frac{p-1}{2}\frac{l_2-1}{2}}p^{\frac{f}{2}},&\mbox{otherwise,}\end{array}\right.\\
 G(\chi^{l_1^{r_1}l_2^{t_2}})=(-1)^{\frac{p^{\frac{1}{2}\varphi(l_2^{r_2-t_2})}+1}{2l_2^{r_2-t_2}}}p^{\frac{f}{2}},\\
 G(\chi^{l_1^{t_1}l_2^{r_2}})=\left\{\begin{array}{ll}(-1)^{\frac{p-1}{2}\frac{l_2-1+1}{2}}p^{\frac{f}{2}-2l_1^{t_1}\varphi(l_2^{r_2})h_1}\left(\frac{a+b\sqrt{-l_1}}{2}\right)^{2l_1^{t_1}\varphi(l_2^{r_2})},&\mbox{if $l_1{\equiv}3\;(\rm{mod}\,8)$;}\\
         (-1)^{\frac{p-1}{2}\frac{l_2-1}{2}+1}p^{\frac{f}{2}},&\mbox{if $l_1{\equiv}7\;(\rm{mod}\,8)$,}\end{array}\right.\\
 G(\chi^{l_1^{r_1}l_2^{r_2}})=-p^{\frac{f}{2}},\\
 G(\chi^{2l_1^{t_1}l_2^{t_2}})=\left\{\begin{array}{ll} p^{\frac{f}{2}},&\mbox{if $\left(\frac{l_2}{l_1}\right)=1$;}\\
          p^{\frac{f}{2}-l_1^{t_1}l_2^{t_2} h_1}(\frac{a+b\sqrt{-l_1}}{2})^{2l_1^{t_1}l_2^{t_2}},&\mbox{if $\left(\frac{l_2}{l_1}\right)=-1$,}.\end{array}\right.\\
 G(\chi^{2l_1^{r_1}l_2^{t_2}})=p^{\frac{f}{2}},\\
 G(\chi^{2l_1^{t_1}l_2^{r_2}})=-p^{\frac{1}{2}(f-l_1^{t_1}\varphi(l_2^{r_2})h_1)}\left(\frac{a+b\sqrt{-l_1}}{2}\right)^{l_1^{t_1}\varphi(l_2^{r_2})},
 \end{array}$$
where $h_{1}$ is the ideal class number of $\mathbb{Q}(\sqrt{-l_1})$
and $a,b\in\mathbb{Z}$ are determined by equations (\ref{(aa)}).
 \end{thm}

\vskip2mm{\large \maltese \quad\underline{\textbf{Case
F}}.}\vskip2mm

$N=4l_1^{r_1}l_2^{r_2}$. For the Gauss sums $G(\chi^{2^il_1^{x}})\ \
(i=0,1,2;\ 0\leqslant{}x\leqslant{}r_1\ (i,x)\neq(2,r_1))$, we have
the following evaluations.

 \begin{thm}
Let $N=4l_1^{r_1}$, $[(\mathbb{Z}/N\mathbb{Z})^*:<{p}>]=2$, i.e.
$f={\varphi(l_1^{r_1})}$. Take $q=p^f$ and $\chi$ be a
multiplicative character of order $N$ over $\mathbb{F}_q$. Assume
that the order of $p$ in group $(\mathbb{Z}/l_1^{r_1}\mathbb{Z})^*$
is $\varphi(l_1^{r_1})/a_1$. Then, for $0\leqslant{}t_1<r_1$,

(1). \underline{Case F1}: ($l_1{\equiv}1\pmod4,\ a_1=1$,
$p{\equiv}3\pmod4,\ K=\mathbb{Q}(\sqrt{-l_1})$)
$$\begin{array}{l}
 G(\chi^{l_1^{t_1}})=p^{\frac{f}{4}}(a+b\sqrt{-l_1})^{l_1^{t_1}},\\
 G(\chi^{l_1^{r_1}})=G(\chi^{2l_1^{t_1}})=G(\chi^{2l_1^{r_1}})=-p^{\frac{f}{2}},\\
 G(\chi^{4l_1^{t_1}})=p^{\frac{f}{2}},
 \end{array}$$
where $a,b\in\mathbb{Z}$ are determined by equations (\ref{(f-3)}).

(2). \underline{Case F2}: ($l_1{\equiv}1\pmod2,\ a_1=1$,
$p{\equiv}1\pmod4,\ K=\mathbb{Q}(\sqrt{-1})$)
$$\begin{array}{l}
 G(\chi^{l_1^{t_1}})=\left\{\begin{array}{ll}\varepsilon p^{\frac{f}{2}},&\mbox{if $\l_1{\equiv}1\pmod4$;}\\
     \varepsilon\sqrt{p}p^{\frac{1}{2}(f-l_1^{t_1}-1)}({a'}+{b'}\sqrt{-1})^{l_1^{t_1}},&\mbox{if $\l_1{\equiv}1\pmod4$,}\end{array}\right.\\
 G(\chi^{l_1^{r_1}})=(-1)^{\frac{1}{8}(\widetilde{a}^2+2\widetilde{b})f+1}\cdot C^fp^{\frac{f}{4}}\cdot(\widetilde{a}+\widetilde{b}\sqrt{-1})^{\frac{f}{2}},\\
 G(\chi^{2l_1^{t_1}})=G(\chi^{2l_1^{r_1}})=-p^{\frac{f}{2}},\\
 G(\chi^{4l_1^{t_1}})=p^{\frac{f}{2}},
 \end{array}$$
where $\varepsilon\in\{\pm1,\pm i\}$ can be determined by
Stickelberger congruence, integers $|a'|,|b'|$ are determined by
$(a')^2+(b')^2=p$ and by the formula (\cite[Thm4.2.3]{B-E-W}) of
quartic Gauss sums over $\mathbb{F}_p$,
$\widetilde{a},\widetilde{b},C\in\mathbb{Z}$ are determined by
$\widetilde{a}^2+\widetilde{b}^2=p,\ \widetilde{a}{\equiv}-1\pmod4,\
C{\equiv}\frac{|\widetilde{b}|}{\widetilde{a}}(\frac{p-1}{2})!\lf{2}{p}\pmod{p}$.

(3). \underline{Case F3}: ($3\neq l_1{\equiv}3\pmod4,\ a_1=2$,
$p{\equiv}3\pmod4,\ K=\mathbb{Q}(\sqrt{-l_1})$)
$$\begin{array}{l}
 G(\chi^{l_1^{t_1}})=\left\{\begin{array}{ll}(-1)^{\frac{p+1}{4}}p^{\frac{f}{2}}&\mbox{if $\l_1{\equiv}7\pmod8$;}\\
     (-1)^{\frac{p+1}{4}}p^{\frac{f}{2}-l_1^{t_1} h_1}\left(\frac{a+b\sqrt{-l_1}}{2}\right)^{l_1^{t_1}}&\mbox{if $\l_1{\equiv}3\pmod8$,}\end{array}\right.\\
 G(\chi^{l_1^{r_1}})=(-1)^{\frac{p+1}{4}}p^{\frac{f}{2}},\\
 G(\chi^{2l_1^{t_1}})=p^{\frac{f}{2}-2l_1^{t_1} h_1}\left(\frac{a+b\sqrt{-l_1}}{2}\right)^{4l_1^{t_1}},\\
 G(\chi^{2l_1^{r_1}})=p^{\frac{f}{2}},\\
 G(\chi^{4l_1^{t_1}})=-p^{\frac{f}{2}-l_1^{t_1} h_1}\left(\frac{a+b\sqrt{-l_1}}{2}\right)^{2l_1^{t_1}},
 \end{array}$$
where $h_{1}$ is the ideal class number of $\mathbb{Q}(\sqrt{-l_1})$
and $a,b\in\mathbb{Z}$ are determined by equations (\ref{(aa)}).
 \end{thm}

\def\refname

%\end{CJK*}
\end{document}